\documentclass[12 pt]{amsart}
\usepackage{amssymb, mathtools, amsfonts, amsthm, graphics,mathrsfs}
\usepackage[usenames, dvipsnames]{xcolor}
\usepackage[hmargin=1 in, vmargin = 1 in]{geometry}
\usepackage{tikz-cd}
\usetikzlibrary{matrix, calc, arrows} 
\usepackage{hyperref}
\usepackage[all]{xy}
\usepackage{ytableau}
\usepackage{soul,dirtytalk}
\usepackage[capitalise]{cleveref}
\usepackage{graphicx,pgf,tikz}

\usepackage{twemojis}
\usepackage{mathdots}

\definecolor{darkblue}{rgb}{0.0,0,0.7}

\newcommand{\newword}[1]{\textcolor{darkblue}{\textbf{\emph{#1}}}}

\newtheorem{Theorem}{Theorem}[section]

\newtheorem{corollary}[Theorem]{Corollary}

\newtheorem{remark}[Theorem]{Remark}

\newtheorem{definition}[Theorem]{Definition}
\newtheorem{proposition}[Theorem]{Proposition}

\newtheorem{theorem*}{Theorem}

\newtheorem{lemma}[Theorem]{Lemma}

\numberwithin{equation}{section}

\newcommand{\rk}{\mathrm{rk}}

\newcommand{\ess}{\mathrm{ess}}
\newcommand{\Mat}{\mathrm{Mat}}
\newcommand{\perm}{\mathrm{Perm}}

\newcommand{\inv}{\mathrm{inv}}

\newcommand{\bpd}{\mathrm{BPD}}

\newcommand{\field}{\kappa}

\newcommand{\lc}{{\rm lc}}

\newcommand{\dc}{{\rm dc}}

\newcommand{\asmtobpd}{\Phi}

\theoremstyle{remark}
\newtheorem{examplex}[Theorem]{Example}
\newenvironment{example}
  {\pushQED{\qed}\examplex}
  {\popQED\endexamplex}

\usepackage[colorinlistoftodos]{todonotes}

\newcommand{\asm}{{\sf ASM}}
\newcommand{\monotone}{{\sf MT}}

\title[Weak order: ASMs, MTs, and BPDs]{Weak order: Alternating sign matrices, monotone triangles, and bumpless pipe dreams}

\author{Laura Escobar}
\address[LE]{Mathematics Department, University of California Santa Cruz, Santa Cruz CA 95064}
\email{lauraescobar@ucsc.edu}

\author{Patricia Klein}
\address[PK]{Department of Mathematics, Texas A\&M University, College Station TX 77840}
\email{pjklein@tamu.edu}

\author{Anna Weigandt}
\address[AW]{School of Mathematics, University of Minnesota, Minneapolis MN 55455}
\email{weigandt@umn.edu}

\thanks{LE was partially funded by NSF CAREER grant DMS-2142656, DMS-2521270. PK was partially funded by NSF grant DMS-2246962 and by the Travel Support for Mathematicians gift MP-TSM-00002939 from the Simons Foundation. AW was partially funded by NSF grant DMS-2344764.
PK worked on this project while she was a member at the Institute for Advanced Study with support from the Bob Moses Fund. She thanks the IAS for its hospitality and support. She also worked on this project while she was a long-term visitor at the Fields Institute, which she also thanks for its hospitality and support.}

\date{\today}

\makeatletter
\@namedef{subjclassname@2020}{%
  \textup{2020} Mathematics Subject Classification}
\makeatother

\subjclass[2020]{}

\begin{document}

\begin{abstract}
In 2018, Hamaker and Reiner introduced weak order for monotone triangles, which extended the usual notion of weak order on the symmetric group. Monotone triangles on $\{1, \ldots, n\}$ are well-known to be in bijection with the set $\asm(n)$ of $n \times n$ alternating sign matrices. Hamaker and Reiner defined weak order on $\asm(n)$ to be induced from weak order on monotone triangles via the standard bijection. Recently, the present authors used an \emph{a priori} different definition of weak order on $\asm(n)$ to give a combinatorial characterization of the codimension of ASM varieties and to show that the natural K-theoretic representatives of these varieties satisfy a divided difference recurrence. 

In the present work, we establish compatibility of these definitions of weak order on $\asm(n)$. Additionally, we give three different explicit means of computing weak order covering relations on $\asm(n)$: on ASMs themselves, on monotone triangles in a manner different from that given by Hamaker and Reiner, and on bumpless pipe dreams, which are a newer family of combinatorial objects also in correspondence with ASMs. Finally, using the language of bumpless pipe dreams, we characterize the fibers of the weak order operators, each of which forms a sublattice of the strong Bruhat order on $\asm(n)$.
\end{abstract}

\maketitle
\tableofcontents

\section{Introduction}\label{sec:intro}

Alternating sign matrices (ASMs) are objects of longstanding interest in enumerative combinatorics and, via the six-vertex ice model, in statistical mechanics.  Weigandt \cite{Wei17} showed that ASMs index arbitrary intersections of matrix Schubert varieties, which are affine varieties introduced by Fulton \cite{Ful92} in the study of Schubert varieties in the complete flag variety.  The multidegrees (resp., twisted K-polynomials) of matrix Schubert varieties are Schubert polynomials \cite{KM05} (resp., Grothendieck polynomials), which are combinatorially natural representatives of Schubert classes in cohomology (resp., K-theory).  In studying these multidegrees and twisted K-polynomials, one is naturally led to study the intersections of matrix Schubert varieties, which are now called ASM varieties.

One robust method for understanding matrix Schubert varieties is induction on strong (Bruhat) order.  One encounters significantly more obstructions when trying to apply similar arguments to ASM varieties.  In order to study the codimensions of ASM varieties as well as their multidegrees and twisted K-polynomials (which we call ASM Schubert polynomials and ASM Grothendieck polynomials, respectively), the authors appealed instead to weak (Bruhat) order \cite{EKW-main}, introduced for ASMs by Hamaker and Reiner \cite{HR20} in the language of monotone triangles.  

One goal of the present paper is to reconnect these two perspectives. Specifically, Hamaker and Reiner take weak order on ASMs to be defined by weak order on monotone triangles, whereas \cite{EKW-main} gives a direct definition on ASMs. \cref{weak_order_isomorphism} establishes compatibility of these two definitions. Collecting correspondences between various features of ASMs and monotone triangles, in \cref{equivalent-descent-conditions} we give a family of conditions that are equivalent to the action of a weak order operator on an ASM being nontrivial. As a consequence, we characterize the ASMs that are maximal in weak order in terms of their essential cells (\cref{prop:maximal-iff-essential-cell-in-every-row}), a companion to Hamaker and Reiner's corresponding characterization for monotone triangles \cite[Lemma 6.2]{HR20}.

Our second focus is computational. We give three explicit means for computing the ASMs covered by an ASM $A$ in weak order. The first is directly from the entries of $A$ (\cref{prop:toggle-one-block-in-weak-order}). The second is via a weak order that respects columns rather than rows (\cref{prop:monotone_transpose_weak_order}), which passes through the correspondence with monotone triangles. The third (\cref{prop:weakorderbpd}) is via combinatorial objects called bumpless pipe dreams, introduced by Lam, Lee, and Shimozono \cite{LLS21} and shown to be in bijection with ASMs by Weigandt \cite{Wei21}. Using this final description, we describe moves to produce the full lattice of ASMs that share an image under a fixed weak order operator (\cref{prop:moves-to-get-preimage}).

\section{Background}

\subsection{ASMs and ASM varieties} 

We begin by reviewing definitions and some basic constructions concerning alternating sign matrices.  We give a brief introduction here. For further information and examples, see \cite[Section 2]{EKW-main}. See \Cref{background-ex} at the end of this section for an example illustrating much of the terminology we give here.

An \newword{alternating sign matrix (ASM)} $A = (A_{i,j})$ is an $n \times n$ matrix with entries in $\{-1,0,1\}$ such that
\begin{enumerate}
    \item for all $i,m \in [n]$, $\sum_{j=1}^m A_{i,j} \in \{0,1\}$,
    \item for all $j, m \in [n]$, $\sum_{i=1}^m A_{i,j} \in \{0,1\}$, and 
    \item $\sum_{(i,j) \in [n] \times [n]} A_{i,j} = n$.
\end{enumerate}  We write $\asm(n)$ for the set of $n\times n$ alternating sign matrices.  Imposing conditions (1) and (2) is equivalent to insisting that the nonzero entries in each row and column alternate in sign with the first nonzero entry being a $1$. Adding condition (3) then requires the sum along each row and column to be $1$.

The set $S_n$ of $n\times n$ permutation matrices is contained in $\asm(n)$.
We identify permutation matrices with permutations of $[n]$ as follows. 
Suppose that the rows of the permutation matrix $w$ are the standard basis vectors $e_{i_1},\ldots,e_{i_n}$. The permutation associated to $w$ maps $j\mapsto i_j$. We use $w\in S_n$ to denote both the permutation matrix and the corresponding permutation of $[n]$.

For $A \in\asm(n)$, define its \newword{corner sum function} $\rk_A$ by $\rk_A(i,j)=\sum_{a\in[i],b\in[j]} A_{a,b}$ for all $i,j\in[n]$. 
We will use the corner sum function for two purposes: (1) to associate to $A$ a subvariety of affine $n^2$-space and (2) to define strong and weak Bruhat orders.  We will often conflate the corner sum function $\rk_A$ with the matrix $(\rk_A(i,j))_{1\leq i,j\leq n}$, writing $\rk_A$ for both.

Fix a field $\field$. Let $\Mat(n)$ denote the set of $n\times n$ matrices over $\field$.  Given $M\in \Mat(n)$, we write $M_{[i],[j]}$ for the submatrix of $M$ consisting of its first $i$ rows and $j$ columns.
The \newword{ASM variety} of $A\in\asm(n)$ is 
\[
X_A=\{M\in \Mat(n) : \rk(M_{[i],[j]})\le \rk_A(i,j) \text{ for all } i,j\in[n]\}.
\]
Let $S=\field [z_{i,j} : i,j\in[n]]$, and let $Z=(z_{i,j})_{i,j\in[n]}$ be a generic $n \times n$ matrix.
Let $I_k(Z_{[i],[j]})\subseteq S$ be the ideal generated by the $k$-minors of $Z_{[i],[j]}$.  Define
\[
I_A = \sum_{i,j\in[n]}I_{\rk_A(i,j)+1}(Z_{[i],[j]}), 
\] which we call the \newword{ASM ideal} of $A$.  It is easy to see that $X_A = \mathbb{V}(I_A)$. By \cite[Section 7.2]{Knu09} and \cite[Lemma 5.9]{Wei17}, or \cite[Lemma 2.6]{KW23}, $I_A$ is  radical.  For the classical case, when $A \in S_n$ and $X_A$ is called a matrix Schubert variety, we refer the reader to \cite{Ful92, KM05}.  

We define \newword{strong (Bruhat) order} on $\asm(n)$ by the relation $A \leq B$ if and only if $\rk_A(i,j)\geq \rk_B(i,j)$ for all $i,j\in [n]$ (or, equivalently, if $X_A\supseteq X_B$, or, equivalently, if $I_A\subseteq I_B$).  Strong order can be used to characterize the components of an ASM variety. For $A \in \asm(n)$, let \[
\perm(A)= \{w\in S_n: w\ge A \text{ and, if } w\ge v\ge A \text{ for some } v\in S_n,\text{ then }w=v\},
\] which we call the \newword{permutation set of $A$}.

\begin{proposition}\cite[Proposition 5.4]{Wei17}\label{prop:I_A-intersection-of-schubs-in-perm}
Let $A \in \asm(n)$.  Then $I_A = \bigcap_{w \in \perm(A)} I_w$ is the minimal prime decomposition of the radical ideal $I_A$.
\end{proposition}

Fulton \cite[Theorem 3.3]{Ful92} showed that $I_A$ is prime whenever $A \in S_n$.  Combining \cite[Proposition 5.4]{Wei17} and \cite[Theorem 3.3]{Ful92}, we see that $I_A$ is prime if and only if $A \in S_n$.

\subsection{Weak (Bruhat) order}
In order to define weak order on $\asm(n)$, we will need the operators $\pi_i$.  

\begin{definition}
\label{def:pi_i}
    Given $A\in \asm(n)$ and $i\in[n-1]$, let 
    \[
    \pi_i(A)=\min\{B\in \asm(n):\rk_A(a,b)=\rk_B(a,b) \text{ for all } a,b\in [n] \text{ with } a\neq i\},
\] where the minimum is taken in the strong order poset.
\end{definition}

Though it is not clear a priori that $\pi_i(A)$ is well-defined, i.e., that a unique minimum element of the given set exists, this follows from \cite[Lemma 2.9]{EKW-main}.  This definition of $\pi_i(A)$ agrees with that of \cite{HR20}, which is given in terms of monotone triangles.

Define \newword{weak (Bruhat) order} on $\asm(n)$ to be the transitive closure of the following covering relations: $A$ covers $B$ if $B = \pi_i(A)$ for some $i \in [n-1]$ and $B \neq A$.  Reserving $C <A$ to denote that $C$ is less than $A$ in strong order, we use $C \prec A$ to denote that $C$ is less than $A$ in weak order.  Note that $\pi_i(A) \leq A$, which implies that the transitive closure described above is well-defined.

We say that $i \in [n-1]$ is a \newword{descent} of $A$ (or that $A$ has a descent at $i$) if $\pi_i(A)<A$.  We say that $i \in [n]$ is an \newword{ascent} of $A$ (or that $A$ has an ascent at $i$) if $i=n$ or if $\pi_i(A) = A$.

The descent set of $A$ can be characterized directly from the rank function, via \newword{essential cells}, the set of which we call the \newword{essential set}, denoted $\ess(A)$.  We define $\ess(A)$ to be the set of cells $(i,j)$ such that
\[
\rk_A(i,j)=\rk_A(i,j-1)=\rk_A(i-1,j)=\rk_A(i,j+1)-1=\rk_A(i+1,j)-1.
\]

\begin{lemma}{\cite[Lemma 3.2]{EKW-main}}\label{lemma:descentessential}
    Let $A\in\asm(n)$.  Then $i$ is a descent of $A$ if and only if there is an essential cell in row $i$ of $A$.
\end{lemma}

We may also understand the essential set in terms of the inversion set of $A$, defined as follows.  We say that $(i,j)$ is an \newword{inversion} of $A$ if 
\[\sum_{k=1}^jA_{i,k}=\sum_{l=1}^i A_{l,j}=0.\]
We write $\inv(A)$ for the set of inversions of $A$.  The essential set of $A$ may also be characterized as
\[
\ess(A)=\{(i,j)\in \inv(A) : (i + 1, j), (i, j + 1) \notin \inv(A)\}.
\]

\begin{example}\label{background-ex}
Let $A = \begin{pmatrix}
    0 & 0 & 1 & 0\\
    1 & 0 & -1 & 1\\
    0 & 0 & 1 & 0\\
    0 & 1 & 0 & 0
\end{pmatrix} \in \asm(4)$.  Its rank function, recorded as a matrix, is $\rk_A = \begin{pmatrix}
    0 & 0 & 1 & 1\\
    1 & 1 & 1 & 2\\
    1 & 1 & 2 & 3\\
    1 & 2 & 3 & 4
\end{pmatrix}$.  
We often picture an element of $\asm(n)$ as an $n \times n$ grid with $\bullet$ representing a $1$ and a $\circ$ representing $-1$. If we draw a line emanating right and down out of each $\bullet$ and stopping at a $\circ$ or at the boundary of the grid, the elements of $\inv(A)$ are exactly the blank tiles together with those containing a $\circ$.  In this example, the corresponding grid is 
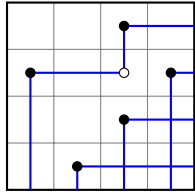
\[
\begin{tikzpicture}[x=1.5em,y=1.5em]
\draw[step=1,gray, thin] (0,0) grid (4,4);
\draw[color=black, thick] (0,0) rectangle (4,4);
\draw[thick, color=blue] (0.5,0) -- (0.5,2.5)-- (2.4,2.5); 
\draw[thick, color=blue] (1.5,0) -- (1.5,0.5) -- (4,0.5); 
\draw[thick, color=blue] (2.5,0) -- (2.5,1.5) -- (4,1.5); 
\draw[thick, color=blue] (2.5,2.6) -- (2.5,3.5) -- (4,3.5); 
\draw[thick, color=blue] (3.5,0) -- (3.5,2.5)-- (4,2.5); 
\filldraw[black] (0.5,2.5) circle (.1); 
\filldraw[black] (1.5,.5) circle (.1); 
\filldraw[black] (2.5,1.5) circle (.1); 
\draw[black] (2.5,2.5) circle (.1); 
\filldraw[black] (2.5,3.5) circle (.1); 
\filldraw[black] (3.5,2.5) circle (.1); 
\end{tikzpicture}.
\] We call this grid the \newword{Rothe diagram} of $A$.  Then $\inv(A) = \{(1,1), (1,2), (2,3), (3,2)\}$.  The elements of $\ess(A)$ are the southeastmost elements of the regions of $\inv(A)$ that are connected by edges of the grid, i.e., $\ess(A) = \{(1,2), (2,3), (3,2)\}$. Because $\ess(A)$ has an element in rows $1$, $2$, and $3$, the descent set of $A$ is $\{1,2,3\}$. Continuing to record ASMs as marked grids, we have 
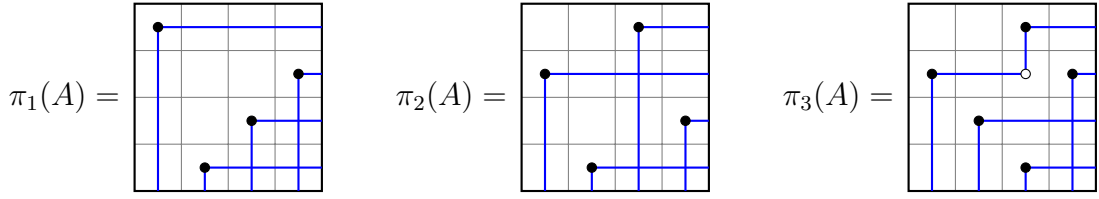
\[
\begin{tikzpicture}[x=1.5em,y=1.5em]
\draw[step=1,gray, thin] (0,0) grid (4,4);
\draw[color=black, thick] (0,0) rectangle (4,4);
\draw[thick, color=blue] (0.5,0) -- (0.5,3.5)-- (4,3.5); 
\draw[thick, color=blue] (1.5,0) -- (1.5,0.5) -- (4,0.5); 
\draw[thick, color=blue] (2.5,0) -- (2.5,1.5) -- (4,1.5); 
\draw[thick, color=blue] (3.5,0) -- (3.5,2.5)-- (4,2.5); 
\filldraw[black] (0.5,3.5) circle (.1); 
\filldraw[black] (1.5,.5) circle (.1); 
\filldraw[black] (2.5,1.5) circle (.1); 
\filldraw[black] (3.5,2.5) circle (.1); 
\node at (-1.5,2){$\pi_1(A) =$};
\end{tikzpicture} \qquad \begin{tikzpicture}[x=1.5em,y=1.5em]
\draw[step=1,gray, thin] (0,0) grid (4,4);
\draw[color=black, thick] (0,0) rectangle (4,4);
\draw[thick, color=blue] (0.5,0) -- (0.5,2.5)-- (4,2.5); 
\draw[thick, color=blue] (1.5,0) -- (1.5,0.5) -- (4,0.5); 
\draw[thick, color=blue] (2.5,0) -- (2.5,3.5) -- (4,3.5); 
\draw[thick, color=blue] (3.5,0) -- (3.5,1.5)-- (4,1.5); 
\filldraw[black] (0.5,2.5) circle (.1); 
\filldraw[black] (1.5,.5) circle (.1); 
\filldraw[black] (2.5,3.5) circle (.1); 
\filldraw[black] (3.5,1.5) circle (.1); 
\node at (-1.5,2){$\pi_2(A) =$};
\end{tikzpicture} \qquad \begin{tikzpicture}[x=1.5em,y=1.5em]
\draw[step=1,gray, thin] (0,0) grid (4,4);
\draw[color=black, thick] (0,0) rectangle (4,4);
\draw[thick, color=blue] (0.5,0) -- (0.5,2.5)-- (2.4,2.5); 
\draw[thick, color=blue] (2.5,0) -- (2.5,0.5) -- (4,0.5); 
\draw[thick, color=blue] (1.5,0) -- (1.5,1.5) -- (4,1.5); 
\draw[thick, color=blue] (2.5,2.6) -- (2.5,3.5) -- (4,3.5); 
\draw[thick, color=blue] (3.5,0) -- (3.5,2.5)-- (4,2.5); 
\filldraw[black] (0.5,2.5) circle (.1); 
\filldraw[black] (2.5,.5) circle (.1); 
\filldraw[black] (1.5,1.5) circle (.1); 
\draw[black] (2.5,2.5) circle (.1); 
\filldraw[black] (2.5,3.5) circle (.1); 
\filldraw[black] (3.5,2.5) circle (.1); 
\node at (-1.5,2){$\pi_3(A) =$};
\end{tikzpicture}.
\] We do not expect that the reader will want to verify these three computations of various $\pi_i(A)$ from the definition, though they may choose to in the course of reading \Cref{sect:computing}.

Regarding the ASM ideal, we compute \begin{align*}
I_A & = \left( z_{11}, z_{12}, z_{13}z_{21}, z_{13}z_{22}, \begin{vmatrix} 
z_{21} & z_{22}\\
z_{31} & z_{32}
\end{vmatrix}\right) \\
& = \left( z_{11}, z_{12}, z_{13}, \begin{vmatrix} 
z_{21} & z_{22}\\
z_{31} & z_{32}
\end{vmatrix}\right) \cap (z_{11}, z_{12}, z_{21}, z_{22})\\
& = I_{4132} \cap I_{3412}.
\end{align*} 
The minimal prime decomposition $I_A = I_{4132} \cap I_{3412}$ reflects the fact that 
\[\perm(A) = \{4132, 3412\}. \qedhere\]
\end{example}

\section{Reconnecting with monotone triangles}
\label{section:monotone_triangles}

The goal of this section is to describe the relationship between weak order on ASMs, as described in \cite{EKW-main} and in the present paper, and the characterization given by Hamaker and Reiner \cite{HR20} in terms of combinatorially equivalent objects called monotone triangles.  
We begin by recalling the relevant definitions.  
A \newword{monotone triangle} of size $n$ is a triangular array $m=(m(i,j))_{1 \leq j \leq i \leq n}$ with entries in $[n]$ satisfying the following rules:
\begin{equation}
\label{eq:monotone_conditions}
    m(i,j)<m(i,j+1),\ m(i,j)\leq m(i-1,j), \text{ and } m(i,j)\leq m(i+1,j+1)
\end{equation}
whenever both sides of the inequalities are defined.  
We arrange this array as pictured below.
\[
\begin{array}{ccccc}
m(1,1) &  &  &  & \\
m(2,1) & m(2,2) &   & & \\
 \vdots &  \vdots  &  \ddots  & & \\
m(n,1) & m(n,2)& \cdots & m(n,n) \\
\end{array}
\]
The conditions in \cref{eq:monotone_conditions} say that entries strictly increase along rows, weakly increase up columns, and weakly increase down diagonals.  This definition is equivalent to the one in \cite{HR20} in terms of subsets that \emph{interlace}. Write $\monotone(n)$ for the set of monotone triangles of size $n$.  

Mills, Robbins, and Rumsey \cite{MRR83} gave a bijection from alternating sign matrices to monotone triangles; for a proof, see, e.g., \cite[Proposition~3.4]{BSS}.  Given $A\in \asm(n)$, we form $m_A\in \monotone(n)$ as follows.  Let $\widetilde{A}$ be the matrix whose $i$th row is the sum of the first $i$ rows of $A$.  Then $\widetilde{A}$ necessarily consists of $0$'s and $1$'s, with exactly $i$ entries equal to $1$ in row $i$.  
We let $m_A(i,1)<m_A(i,2)<\cdots <m_A(i,i)$ be the ordered list of the column indices of the $1$'s in row $i$ of $\widetilde{A}$. 

\begin{example}
\label{ex: asm to monotone}
      Let 
  $A=\begin{pmatrix}
0  &0 & 1 & 0\\
 0 & 1& -1 & 1\\
 1 & 0 & 0 & 0\\
 0 & 0 & 1 & 0
    \end{pmatrix}.$  
Then $\widetilde{A} = \begin{pmatrix}
0 & 0 & 1 & 0 \\
0 & 1 & 0 & 1 \\
1 & 1 & 0 & 1 \\
1 & 1 & 1 & 1
\end{pmatrix}$ and so
\[
m_A =
\begin{array}{cccc}
3\\
2&4\\
1&2&4\\
1&2&3&4
\end{array}.\qedhere
\]
\end{example}

We say that $m\in \monotone(n)$ is a \newword{key} if the entries in row $i$ of $m$ form a subset of the entries in row $i+1$ of $m$ for all $i\in[n-1]$.  
The bijection between ASMs and monotone triangles restricts to a bijection from permutation matrices to keys.

Following \cite{LS96}, we define \newword{strong (Bruhat) order} on $\monotone(n)$ in the following way: We say \[m_A\leq m_B \text{ if } m_A(i,j)\leq m_B(i,j) \text{ for all } 1\leq j\leq i\leq n.\]  
For $w,v\in S_n$, we have $w\leq v$ if and only if $m_w\leq m_v$; this is the Ehresmann criterion for Bruhat order (see \cite[Theorem 2.6.3]{BB05}). 
 Indeed, corner sum matrices and monotone triangles induce the same lattice structure on ASMs, which we will show momentarily in \cref{lemma:monotone_and_rank_induce_same_poset}.

Before continuing, we record an observation, which we will use several times below. Fix $A \in \asm(n)$ and $i,j \in [n]$. Then 
\begin{equation}\label{eq:rank-from-triangle}
    \rk_A(i,j) = \sum_{k=1}^j \widetilde{A}_{i,k} = \#\{\ell \in [i] : m_A(i,\ell) \leq j\}.
\end{equation}

\begin{lemma}
\label{lemma:monotone_and_rank_induce_same_poset}
    Given $A,B\in \asm(n)$, $m_A\leq m_B$ if and only if $A \leq B$. Hence $\asm(n)$ and $\monotone(n)$ are isomorphic as posets under strong order.
\end{lemma}
\begin{proof}
Let $\widetilde{A}$ be as in the bijection from \cite{MRR83}, as recalled above.  
If $m_A\leq m_B$, then for all $i,j\in [n]$, by \cref{eq:rank-from-triangle},
        \begin{align*}
        \rk_A(i,j)&= \#\{\ell \in [i] : m_A(i,\ell) \leq j\}\\
        &\geq \#\{\ell \in [i] : m_B(i,\ell) \leq j\}\\
        &=\rk_B(i,j).
    \end{align*}
Thus, by the definition of strong order, $A \leq B$.  

Conversely, suppose that $m_A\not \leq m_B$.  Then we may fix $i \in [n]$, $j \in [i]$ such that $m_A(i,j)>m_B(i,j)$, and set $t = m_B(i,j)$. 
 Then $\#\{\ell \in [i] : m_B(i,\ell) \leq t\} = j$ while $\#\{\ell \in [i] : m_A(i,\ell) \leq t\} \leq j-1$.  Hence $\rk_B(i,t) = j>\rk_A(i,t)$, and so $A\not \leq B$.
\end{proof}

In view of \cref{lemma:monotone_and_rank_induce_same_poset} and the fact that $\asm(n)$ is a lattice \cite{LS96}, we know that $\monotone(n)$ is a lattice. We record explicit formulas for meets and joins below. 

\begin{lemma}
\label{jointmeetMT}
    The set $\monotone(n)$ is a lattice.  Furthermore, given $a,b\in \monotone(n)$,
    \[(a\vee b)(i,j)=\max(a(i,j),b(i,j))\] and
    \[(a\wedge b)(i,j)=\min(a(i,j),b(i,j))\]
    for all $i,j$ with $1\leq j\leq i\leq n$.
\end{lemma}
\begin{proof}
    One may verify this from the defining inequalities for $\monotone(n)$ (as given in \cref{eq:monotone_conditions}) by confirming that taking pointwise minima and maxima produces a valid element of $\monotone(n)$; we omit the details.
\end{proof}

The covering relations in $\monotone(n)$ under strong order admit a natural description in terms of the entries of $m$. Fix $m\in \monotone(n)$, $a \in [n]$, and $b \in [a]$.  Form the array $m'$ defined by 
\[m'(i,j)=\begin{cases}
    m(a,b)-1 &\text{if } (i,j)=(a,b)\\
    m(i,j) &\text{otherwise.}
\end{cases}
\] 
We call $(a,b)$ an \newword{essential point} of the monotone triangle $m\in \monotone(n)$ if $m'\in \monotone(n)$. It is clear that $m$ covers any such $m'$. By \cite[Corollary 6.7]{For08}, these are the only covering relations.

\begin{example}
    The essential points of the monotone triangle from \cref{ex: asm to monotone} are $(1,1)$, $(2,1)$, and $(2,2)$.
\end{example}

\begin{lemma}
\label{lemma:monotone_essential_set_translation}
    Let $A\in \asm(n)$.  Then $(a,b)$ is an essential point of $m_A$ if and only if $(a,m_A(a,b)-1)\in \ess(A)$.
\end{lemma}

\begin{proof}

\noindent $(\Rightarrow)$ Suppose that $(a,b)$ is an essential point of $m_A$.  For convenience, set $j=m_A(a,b)-1$.  By the definition of essential point, we have $m_A(a,b-1)< j$ and also $m_A(a-1,b-1),m_A(a+1,b)\leq j$.

Because $m_A(a,b)=j+1$, we have $\rk_A(a,j+1)=b$ and also $\rk_A(a,j)=b-1$ from \cref{eq:rank-from-triangle}.  Set $k = m_A(a,b-1)<j$.  Then, by \cref{eq:rank-from-triangle}, $\rk_A(a,k)=b-1$.  The inequalities $k \leq j-1<j$ together with the equalities $\rk_A(a,k)=b-1 = \rk_A(a,j)$ force $\rk_A(a,j-1)=b-1$ as well.  Thus, \begin{equation}\label{essential-equality-1}
    \rk_A(a,j) = \rk_A(a,j-1) \qquad \mbox{ and } \qquad \rk_A(a,j) = \rk_A(a,j+1)-1.
\end{equation}

Because $m_A(a-1,b-1)\leq j$, $b-1$ first appears weakly before column $j$ in row $a-1$ of $\rk_A$.  In particular, because $\rk_A(a,j)=b-1\geq \rk_A(a-1,j)\geq b-1$, we must have $\rk_A(a-1,j)=b-1$. Thus, \begin{equation}\label{essential-equality-2}
    \rk_A(a,j) = \rk_A(a-1,j).
\end{equation}

Finally, because $\rk_A(a+1,j)-\rk_A(a,j) \in \{0,1\}$ from the definitions of ASM and rank function, and because $\rk_A(a,j)=b-1$, we have $\rk_A(a+1,j)\in\{b-1,b\}$.  Set $p = m_A(a+1,b)\leq j$. Then by \cref{eq:rank-from-triangle}, $\rk_A(a+1,p) = b$, and so $\rk_A(a+1,j)\geq b$.  Combining $\rk_A(a+1,j)\in\{b-1,b\}$ with $\rk_A(a+1,j)\geq b$, we have $\rk_A(a+1,j) = b$.

Thus, \begin{equation}\label{essential-equality-3}
    \rk_A(a,j) = \rk_A(a+1,j)-1.
\end{equation}

Combining \cref{essential-equality-1}, (\ref{essential-equality-2}), and (\ref{essential-equality-3}), we have $(a,j) \in \ess(A)$ by the definition of essential cell.

\noindent $(\Leftarrow)$ Suppose that $(a,j)\in \ess(A)$. Let $b=\rk_A(a,j)+1$. We will show that $j=m_A(a,b)-1$ and that $(a,b)$ is an essential point of $m_A$.  

From the definition of essential cell, we know that $\rk_A(a,j-1)=b-1$ and $\rk_A(a,j+1)=b$.  Then by \cref{eq:rank-from-triangle}, $m_A(a,b)=j+1$.  

To show that $(a,b)$ is an essential point of $m_A$, we must verify that decreasing $m_A(a,b)$ by $1$ results in a valid monotone triangle.  From the definition of essential cell, we have that $\rk_A(a,j-1)=b-1$, $\rk_A(a-1,j)=b-1$, and $\rk_A(a+1,j)=b$.  Translating these rank conditions to conditions on $m_A$ via \cref{eq:rank-from-triangle} shows $m_A(a,b-1)\leq j-1$, $m_A(a-1,b-1)\leq j$, and $m_A(a+1,b)\leq j$, respectively. Thus, replacing $m_A(a,b)$ by $m_A(a,b)-1$ yields a valid monotone triangle, and so $(a,b)$ is an essential point of $m_A$.
\end{proof}

We recall the definition of weak order and descent for monotone triangles from \cite{HR20}. For $m \in \monotone(n)$ and $i \in [n-1]$, let $\pi_i(m)$ denote the monotone triangle whose entries in rows $j \neq i$ agree with the entries in $m$ and whose entries in row $i$ are minimal among all elements of $\monotone(n)$ subject to that condition. By \cite{HR20}, $\pi_i(m)$ exists and is well defined.  Declare $\pi_i(m) \prec m$ to be a covering relation provided $\pi_i(m) \neq m$, and let the \newword{weak order} on $\monotone(n)$ be the transitive closure of those covering relations. We call $i$ a \newword{descent} of $m$ if $\pi_i(m) \neq m$, in which case $\pi_i(m)<m$ (in strong order).

\begin{lemma}
    \label{descentessentialMT}
    Let $m \in \monotone(n)$. Then $i$ is a descent of $m$ if and only if $m$ has an essential point in row $i$.
\end{lemma}
\begin{proof}
    Let $A \in \asm(n)$ be such that $m = m_A$. By \cref{lemma:monotone_and_rank_induce_same_poset}, $i$ is a descent of $m$ if and only if $i$ is a descent of $A$. By \cite[Lemma 3.2]{EKW-main}, $i$ is a descent of $A$ if and only if $A$ has an essential cell in row $i$. By \cref{lemma:monotone_essential_set_translation}, $A$ has an essential cell in row $i$ if and only if $m$ has an essential point in row $i$. 
\end{proof}

In \cite{EKW-main}, the authors showed that the operators $\pi_i$ distribute over joins. We will now show the same for monotone triangles, which will allow us to conclude that $\asm(n)$ and $\monotone(n)$ are isomorphic as posets under weak order. Said otherwise, this allows us to conclude that the definition of a descent of $A \in \asm(n)$ given in \cite{EKW-main} agrees with the one given in \cite{HR20}. We recall the key lemma from \cite{EKW-main} and then give its analogue for monotone triangles.

\begin{lemma}[\cite{EKW-main}, Lemma 3.3]
\label{lemma:joindescent}
    Suppose $A,B,C\in \asm(n)$ such that $A=B\vee C$.  If $i$ is a descent of $A$, then $i$ is a descent of $B$ or $i$ is a descent of $C$.  
\end{lemma}

\begin{lemma}
\label{lemma:joindescentMT}
    Suppose $m_A,m_B,m_C\in \monotone(n)$ such that $m_A=m_B\vee m_C$.  Let $i \in [n-1]$. If $i$ is a descent of $m_A$, then $i$ is a descent of $m_B$ or $i$ is a descent of $m_C$.  
\end{lemma}
\begin{proof}
  Suppose $i$ is a descent of $m_A$.  Then by \cref{descentessentialMT}, $m_A$ has an essential point in row $i$, say at $(i,j)$. 
  
  Let $A, B, C \in \asm(n)$ correspond to $m_A$, $m_B$, and $m_C$, respectively. By \cref{lemma:monotone_essential_set_translation}, $A$ has an essential cell at $(i,m_A(i,j)-1)$. Hence by \cref{lemma:descentessential}, $i$ is a descent of $A$. Now by \cref{lemma:joindescent}, $i$ is a descent of $B$ or $i$ is a descent of $C$. Then, using \cref{lemma:descentessential} again, $B$ has an essential cell in row $i$ or $C$ has an essential cell in row $i$, whence, by \cref{lemma:monotone_essential_set_translation}, $m_B$ or $m_C$ has an essential point in row $i$. Finally, by \cref{descentessentialMT} again, $i$ is a descent of $m_B$ or $i$ is a descent of $m_C$.
\end{proof}

The proof of the next proposition follows \cite[Proposition 3.4]{EKW-main}, which is the corresponding statement for ASMs.

\begin{proposition}\label{prop:joindescentpart3MT}
    Suppose $m,m_1,\ldots,m_k\in \monotone(n)$ such that $m=m_1\vee \cdots \vee m_k$.  Let $i\in [n-1]$.  Then $\pi_i(m)=\pi_i(m_1)\vee \cdots \vee \pi_i(m_k)$.  
\end{proposition}
\begin{proof}
We proceed by induction on $k$.  If $k=1$, there is nothing to show.  Suppose $k=2$. For convenience, we write $m'=\pi_i(m_1)\vee \pi_i(m_2)$.  It follows from \cref{jointmeetMT} that $m$ agrees with $m'$ outside of row $i$.  As such, $\pi_i(m')=\pi_i(m)$.  If $m' \neq \pi_i(m)$, then \cref{lemma:joindescentMT} implies that $i$ is a descent of $\pi_i(m_1)$ or of $\pi_i(m_2)$, which contradicts \cref{descentessentialMT}.  Thus, $m'=\pi_i(m)$ as desired.

For $k>2$, the result follows from a straightforward induction.
\end{proof}

\begin{proposition}
    \label{weak_order_isomorphism}
    Given $A \in \asm(n)$ and $i \in [n-1]$, $\pi_i(m_A) = m_{\pi_i(A)}$. Hence $\asm(n) \cong \monotone(n)$ as posets under weak order.
    \end{proposition}
\begin{proof}
If $A \in S_n$, we have that $m_A$ is a key. Let $a$ be the element of row $i$ that is not an element of row $i-1$ and $b$ be the element of row $i+1$ that is not an element of row $i$. Then row $i$ is minimized, subject to the conditions of interlacing with both rows $i-1$ and $i+1$, by replacing $a$ by $\min\{a,b\}$. One sees easily that the monotone triangle obtained in this way is that corresponding to $\pi_i(A)$, which is $As_i$ if $a>b$ and $A$ if $b>a$.

Now suppose $A\not \in S_n$.
    Using \cite[Theorem 4.4]{LS96}, write $A = w_1 \vee \cdots \vee w_k$,  with $w_j \in S_n$. By \cref{lemma:monotone_and_rank_induce_same_poset}, $m_A = m_{w_1} \vee \cdots \vee m_{w_k}$. Then by \cref{prop:joindescentpart3MT} and \cite[Proposition 3.4]{EKW-main},
    \begin{align*}
\pi_i(m_A) &= \pi_i(m_{w_1}) \vee \cdots \vee \pi_i(m_{w_k}) \\
&= m_{\pi_i({w_1})} \vee \cdots \vee m_{\pi_i(w_k)}\\
& = m_{\pi_i({w_1}) \vee \cdots \vee \pi_i(w_k)}\\
& = m_{\pi_i(w_1 \vee \cdots \vee w_k)}\\
& = m_{\pi_i(A)}. \qedhere
    \end{align*} 
\end{proof}

\section{Computing $\pi_i(A)$}
\label{sect:computing}

Our next goal is to describe the computation of  $\pi_i(A)$ in three ways. The first is straightforward: Using \cref{weak_order_isomorphism}, one option is to compute $\pi_i(m_A)$, as described in \cite{HR20}.  The second is the focus of \cref{sect:computing-pi_i(A)}, which is a direct argument on the ASM $A$ and serves as a primer for the pipe dream combinatorics to appear in \cref{subsection:bpdweakorder}. The third is a different monotone triangle computation, which runs through $m_{A^T}$ and which is the focus of \cref{sect:computing-pi_i(A)-with-column-weak-order}. 

\subsection{Direct computation of $\pi_i(A)$}
\label{sect:computing-pi_i(A)}
We will now give an algorithm for computing $\pi_i(A)$ directly from the ASM $A$, possibly in several steps.  Our first goal in this direction is \cref{prop:toggle-one-block-in-weak-order}, which will produce from an ASM $A$ with a descent at $i$ an ASM $B$ such that $\pi_i(A) \leq B<A$ (recalling that $<$ denotes a strong order relation). By induction on strong order, repeated application of \cref{prop:toggle-one-block-in-weak-order} will yield $\pi_i(A)$.  

It will not be hard to see that \cref{prop:toggle-one-block-in-weak-order} may be applied simultaneously at all essential cells of $A$ in row $i$.  This will not necessarily yield $\pi_i(A)$ but will, as it turns out, yield an ASM whose only essential cells in row $i$ contain a $-1$ and that require only one more application of \cref{prop:toggle-one-block-in-weak-order} to obtain $\pi_i(A)$.  We will omit this rather tedious proof.

The construction given in the proof of \cref{prop:toggle-one-block-in-weak-order} relies on two lemmas from \cite{EKW-main}.  We recall these lemmas presently.

\begin{lemma}{\cite[Lemma 2.17]{EKW-main}}\label{lemma:new-ASM-from-increasing-at-essential-cell}
    Let $A \in \asm(n)$, and suppose that $(i,j) \in \ess(A)$.  Let $f:[n] \times [n] \rightarrow [0,n]$ be the function 
\[f(a,b) = \begin{cases}
 \rk_A(a,b) & \text{if } (a,b) \neq (i,j)\\
 \rk_A(i,j)+1 & \text{if } (a,b) = (i,j).
\end{cases}\]
    Then there exists $B \in \asm(n)$ such that $\rk_B = f$.  Specifically, \[
B_{a,b} = \begin{cases}
A_{a,b}, &\text{ if } (a,b) \notin \{(i,j),(i+1,j),(i,j+1),(i+1,j+1)\} \\
A_{a,b}+1, &\text{ if } (a,b) \in \{(i,j), (i+1,j+1)\}\\
A_{a,b}-1, &\text{ if } (a,b) \in \{(i+1,j), (i,j+1)\}.
\end{cases}
    \]
\end{lemma}

\begin{lemma}{\cite[Lemma 2.19]{EKW-main}}\label{lem:cover-by-add-to-ess-cell}
    Fix $A,B\in \asm(n)$.  Then $A$ covers $B$ in strong order if and only if there exists $(i,j)\in\ess(A)$ such that $\rk_A(a,b)=\rk_B(a,b)$ for all $(a,b)\neq (i,j)$ and $\rk_A(i,j)+1=\rk_B(i,j)$.  
\end{lemma}

 While computing $\pi_i(A)$, the entries of interest to us will be those elements $(i,j)$ of $\inv(A)$ such that $(i+1,j) \notin \inv(A)$.  We name this set for convenience during the proof that follows.

\begin{definition}
    Fix $A \in \asm(n)$ and $i \in [n-1]$.  Let \[
    F_i(A) = \{(i,j) \in \inv(A) : (i+1,j)\notin \inv(A)\}.
\]
\end{definition}

Note that the elements of $\ess(A)$ in row $i$ form a subset of $F_i(A)$. 

\begin{proposition}\label{prop:toggle-one-block-in-weak-order}
       Fix $A \in \asm(n)$, and suppose that $i$ is a descent of $A$.  Suppose that $(i,\ell-1) \in \ess(A)$, and let $j$ be the greatest integer such that $j < \ell$, $(i,j) \in F_i(A)$, and $(i,j-1) \notin F_i(A)$.  Define the matrix $B$ by \[
    B_{a,b} = \begin{cases}
        A_{a,b}, &\text{ if } (a,b) \notin \{(i,j), (i+1,j), (i, \ell), (i+1,\ell)\}\\
        A_{a,b}+1, &\text{ if } (a,b) \in \{(i,j), (i+1,\ell)\}\\
        A_{a,b}-1, &\text{ if } (a,b) \in \{(i+1,j), (i,\ell)\}.
    \end{cases} 
    \]  Then $B<A$ and $\pi_i(B) = \pi_i(A)$.
   \end{proposition}

\begin{example}
Before proving \cref{prop:toggle-one-block-in-weak-order}, we give an example of its simultaneous implementation at the essential cells $(3,1)$, $(3,4)$, and $(3,7)$ of $A$ (that is, with these cells each playing the role of $(i,\ell-1)$).  Note that $B \neq \pi_3(A)$ but that its only essential cell in row $i=3$ occurs at $(3,2)$ where $B_{3,2} = -1$.  The cells $(a,b)$ such that $B_{a,b} = A_{a,b}+1$ are shaded in green, and those satisfying $B_{a,b} = A_{a,b}-1$ are shaded in pink. One more iteration of \cref{prop:toggle-one-block-in-weak-order} will yield $\pi_3(A)$.
    \[
\raisebox{2.35cm}{$A$: }\begin{tikzpicture}[x=1.5em,y=1.5em]
\filldraw[color=black, fill=green](0,5)rectangle(1,6);
\filldraw[color=black, fill=green](1,4)rectangle(2,5);
\filldraw[color=black, fill=green](2,5)rectangle(3,6);
\filldraw[color=black, fill=green](4,4)rectangle(5,5);
\filldraw[color=black, fill=green](6,5)rectangle(7,6);
\filldraw[color=black, fill=green](7,4)rectangle(8,5);
\filldraw[color=black, fill=pink](0,4)rectangle(1,5);
\filldraw[color=black, fill=pink](1,5)rectangle(2,6);
\filldraw[color=black, fill=pink](2,4)rectangle(3,5);
\filldraw[color=black, fill=pink](4,5)rectangle(5,6);
\filldraw[color=black, fill=pink](6,4)rectangle(7,5);
\filldraw[color=black, fill=pink](7,5)rectangle(8,6);
\draw[step=1,gray, thin] (0,0) grid (8,8);
\draw[color=black, thick](0,0)rectangle(8,8);
\node at (0.5,4.5){$1$};
\node at (1.5,7.5){$1$};
\node at (2.5,1.5){$1$};
\node at (3.5,0.5){$1$};
\node at (4.5,3.5){$1$};
\node at (4.5,4.5){$-1$};
\node at (4.5,6.5){$1$};
\node at (5.5,2.5){$1$};
\node at (6.5,4.5){$1$};
\node at (7.5,5.5){$1$};
\end{tikzpicture} \hspace{3cm} 
\raisebox{2.35cm}{$B$: } \begin{tikzpicture}[x=1.5em,y=1.5em]
\filldraw[color=black, fill=green](0,5)rectangle(1,6);
\filldraw[color=black, fill=green](1,4)rectangle(2,5);
\filldraw[color=black, fill=green](2,5)rectangle(3,6);
\filldraw[color=black, fill=green](4,4)rectangle(5,5);
\filldraw[color=black, fill=green](6,5)rectangle(7,6);
\filldraw[color=black, fill=green](7,4)rectangle(8,5);
\filldraw[color=black, fill=pink](0,4)rectangle(1,5);
\filldraw[color=black, fill=pink](1,5)rectangle(2,6);
\filldraw[color=black, fill=pink](2,4)rectangle(3,5);
\filldraw[color=black, fill=pink](4,5)rectangle(5,6);
\filldraw[color=black, fill=pink](6,4)rectangle(7,5);
\filldraw[color=black, fill=pink](7,5)rectangle(8,6);
\draw[step=1,gray, thin] (0,0) grid (8,8);
\draw[color=black, thick](0,0)rectangle(8,8);
\node at (0.5,5.5){$1$};
\node at (1.5,7.5){$1$};
\node at (1.5,4.5){$1$};
\node at (1.5,5.5){$-1$};
\node at (2.5,1.5){$1$};
\node at (2.5,5.5){$1$};
\node at (2.5,4.5){$-1$};
\node at (3.5,0.5){$1$};
\node at (4.5,3.5){$1$};
\node at (4.5,5.5){$-1$};
\node at (4.5,6.5){$1$};
\node at (5.5,2.5){$1$};
\node at (6.5,5.5){$1$};
\node at (7.5,4.5){$1$};
\end{tikzpicture} 
\]   

For the reader who prefers a diagrammatic interpretation, and in anticipation of \cref{subsection:bpdweakorder}, we record the same information via Rothe diagrams.

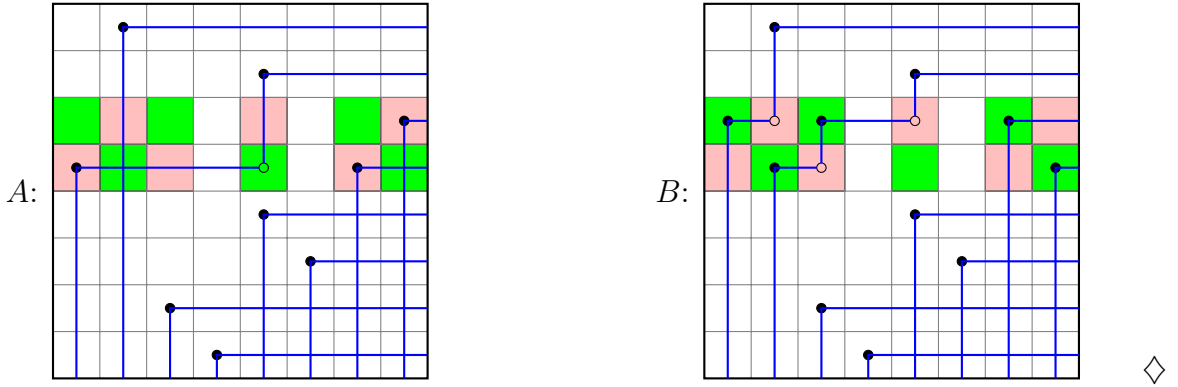
\[
\raisebox{2.35cm}{$A$: }\begin{tikzpicture}[x=1.5em,y=1.5em]
\filldraw[color=black, fill=green](0,5)rectangle(1,6);
\filldraw[color=black, fill=green](1,4)rectangle(2,5);
\filldraw[color=black, fill=green](2,5)rectangle(3,6);
\filldraw[color=black, fill=green](4,4)rectangle(5,5);
\filldraw[color=black, fill=green](6,5)rectangle(7,6);
\filldraw[color=black, fill=green](7,4)rectangle(8,5);
\filldraw[color=black, fill=pink](0,4)rectangle(1,5);
\filldraw[color=black, fill=pink](1,5)rectangle(2,6);
\filldraw[color=black, fill=pink](2,4)rectangle(3,5);
\filldraw[color=black, fill=pink](4,5)rectangle(5,6);
\filldraw[color=black, fill=pink](6,4)rectangle(7,5);
\filldraw[color=black, fill=pink](7,5)rectangle(8,6);
\draw[step=1,gray, thin] (0,0) grid (8,8);
\draw[color=black, thick](0,0)rectangle(8,8);
\filldraw [black](0.5,4.5)circle(.1);
\filldraw [black](1.5,7.5)circle(.1);
\filldraw [black](2.5,1.5)circle(.1);
\filldraw [black](3.5,0.5)circle(.1);
\filldraw [black](4.5,3.5)circle(.1);
\draw [black](4.5,4.5)circle(.1);
\filldraw [black](4.5,6.5)circle(.1);
\filldraw [black](5.5,2.5)circle(.1);
\filldraw [black](6.5,4.5)circle(.1);
\filldraw [black](7.5,5.5)circle(.1);
\draw[thick, color=blue] (.5,0)--(.5,4.5)--(4.4,4.5);
\draw[thick, color=blue] (4.5,4.6)--(4.5,6.5)--(8,6.5);
\draw[thick, color=blue] (1.5,0)--(1.5,7.5)--(8,7.5);
\draw[thick, color=blue] (2.5,0)--(2.5,1.5)--(8,1.5);
\draw[thick, color=blue] (3.5,0)--(3.5,0.5)--(8,0.5);
\draw[thick, color=blue] (4.5,0)--(4.5,3.5)--(8,3.5);
\draw[thick, color=blue] (5.5,0)--(5.5,2.5)--(8,2.5);
\draw[thick, color=blue] (6.5,0)--(6.5,4.5)--(8,4.5);
\draw[thick, color=blue] (7.5,0)--(7.5,5.5)--(8,5.5);
\end{tikzpicture} \hspace{3cm} 
\raisebox{2.35cm}{$B$: } \begin{tikzpicture}[x=1.5em,y=1.5em]
\filldraw[color=black, fill=green](0,5)rectangle(1,6);
\filldraw[color=black, fill=green](1,4)rectangle(2,5);
\filldraw[color=black, fill=green](2,5)rectangle(3,6);
\filldraw[color=black, fill=green](4,4)rectangle(5,5);
\filldraw[color=black, fill=green](6,5)rectangle(7,6);
\filldraw[color=black, fill=green](7,4)rectangle(8,5);
\filldraw[color=black, fill=pink](0,4)rectangle(1,5);
\filldraw[color=black, fill=pink](1,5)rectangle(2,6);
\filldraw[color=black, fill=pink](2,4)rectangle(3,5);
\filldraw[color=black, fill=pink](4,5)rectangle(5,6);
\filldraw[color=black, fill=pink](6,4)rectangle(7,5);
\filldraw[color=black, fill=pink](7,5)rectangle(8,6);
\draw[step=1,gray, thin] (0,0) grid (8,8);
\draw[color=black, thick](0,0)rectangle(8,8);
\filldraw [black](0.5,5.5)circle(.1);
\filldraw [black](1.5,4.5)circle(.1);
\draw [black](1.5,5.5)circle(.1);
\filldraw [black](1.5,7.5)circle(.1);
\filldraw [black](2.5,1.5)circle(.1);
\filldraw [black](2.5,5.5)circle(.1);
\draw [black](2.5,4.5)circle(.1);
\filldraw [black](3.5,0.5)circle(.1);
\filldraw [black](4.5,6.5)circle(.1);
\draw [black](4.5,5.5)circle(.1);
\filldraw [black](4.5,3.5)circle(.1);
\filldraw [black](5.5,2.5)circle(.1);
\filldraw [black](6.5,5.5)circle(.1);
\filldraw [black](7.5,4.5)circle(.1);
\draw[thick, color=blue] (.5,0)--(.5,5.5)--(1.4,5.5);
\draw[thick, color=blue](1.5,5.6)--(1.5,7.5)--(8,7.5);
\draw[thick, color=blue] (1.5,0)--(1.5,4.5)--(2.4,4.5);
\draw[thick, color=blue] (2.5,4.6)--(2.5,5.5)--(4.4,5.5);
\draw[thick, color=blue](4.5,5.6)--(4.5,6.5)--(8,6.5);
\draw[thick, color=blue] (2.5,0)--(2.5,1.5)--(8,1.5);
\draw[thick, color=blue] (3.5,0)--(3.5,0.5)--(8,0.5);
\draw[thick, color=blue] (4.5,0)--(4.5,3.5)--(8,3.5);
\draw[thick, color=blue] (5.5,0)--(5.5,2.5)--(8,2.5);
\draw[thick, color=blue] (6.5,0)--(6.5,5.5)--(8,5.5);
\draw[thick, color=blue] (7.5,0)--(7.5,4.5)--(8,4.5);
\end{tikzpicture} \qedhere
\] 
\end{example}

   \begin{proof}
      We proceed by induction on $\ell-j$.  If $\ell-j = 1$, then $(i,j) = (i,\ell-1) \in \ess(A)$ and so $A$ covers $B$ in strong order by \cref{lem:cover-by-add-to-ess-cell,lemma:new-ASM-from-increasing-at-essential-cell}.  Now $\rk_A(a,b) = \rk_B(a,b)$ for all $a \neq i$, and so $\pi_i(A) = \pi_i(B)$.

          Now with $\ell-j>1$ arbitrary, consider the matrix $B'$ defined by \[
    B'_{a,b} = \begin{cases}
        A_{a,b} \text{ if } (a,b) \notin \{(i,\ell-1), (i+1,\ell-1), (i, \ell), (i+1,\ell)\}\\
        A_{a,b}+1 \text{ if } (a,b) \in \{(i,\ell-1), (i+1,\ell)\}\\
        A_{a,b}-1 \text{ if } (a,b) \in \{(i+1,\ell-1), (i,\ell)\}.
    \end{cases} 
    \]      (For an example of the construction of $B'$ viewed via the Rothe diagram, see \cref{ex:B'-construction}.)  Arguing as above, $B'$ covers $A$ in strong order and $\pi_i(B') = \pi_i(A)$. Our goal is now to apply the inductive hypothesis to $B'$ to show that $\pi_i(B') = \pi_i(B)$ and that $B < B'$. 
    
    Because the matrices $A$ and $B'$ agree weakly northwest of $(i+1,\ell-2)$ and because $(i,\ell-2) \in F_i(A)$, we have $(i,\ell-2) \in F_i(B')$. By construction $(i,\ell-1) \notin \inv(B')$, and so $(i,\ell-2) \in \ess(B')$. Again because the matrices $A$ and $B'$ agree weakly northwest of $(i+1,\ell-2)$ and by the assumption $\ell-j>1$, $j$ is the greatest integer such that $j < \ell-1$, $(i,j) \in F_i(A)$, and $(i,j-1) \notin F_i(A)$.  Obviously $(\ell-1)-j<\ell-j$.  We see directly from the definitions of $B$ and $B'$ that $B_{a,b} = B'_{a,b}$ if $(a,b) \notin \{(i,j),(i+1,j), (i,\ell-1), (i+1,\ell-1)\}$ and that 
    \begin{alignat*}{3}
    B_{i,j} &= A_{i,j}+1 &&= B'_{i,j}+1, \\
    B_{i+1,j} & = A_{i+1,j}-1 &&= B'_{i+1,j}-1,\\
    B_{i,\ell-1} &= A_{i,\ell-1}&& = B'_{i,\ell-1}-1, \text{ and } \\
    B_{i+1,\ell-1}& = A_{i+1,\ell-1}&&= B'_{i+1,\ell-1}+1.
    \end{alignat*}
That is, $B$ is obtained from $B'$ via the construction in the proposition statement for the essential cell $(i, \ell-2) \in \ess(B')$. Thus, by induction, $\pi_i(B') = \pi_i(B)$ and $B<B'$. 
   \end{proof}

   \begin{example}\label{ex:B'-construction}
       In order to illustrate the constructions within the proof of \cref{prop:toggle-one-block-in-weak-order}, we give the Rothe diagrams of $A$, $B'$, and $B$, respectively, corresponding to the choices $i = 3$, $j = 2$, and $\ell = 4$.  The elements of $F_3(B')$ and of $F_3(B)$ appear in yellow.  In light of \Cref{lemma:descentessential}, these boxes in yellow encode the work that remains to be done before arriving at $\pi_3(A)$.  

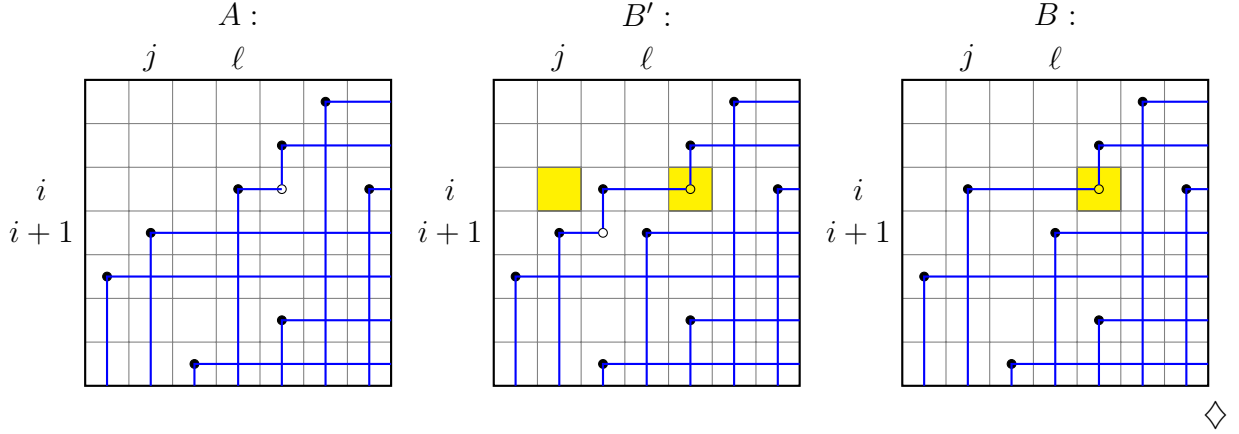
       \[
       \begin{centering}
\begin{tikzpicture}[x=1.4em,y=1.4em]
\draw[step=1,gray, thin] (0,0) grid (7,7);
\draw[color=black, thick](0,0)rectangle(7,7);
\filldraw [black](0.5,2.5)circle(.1);
\filldraw [black](1.5,3.5)circle(.1);
\filldraw [black](2.5,0.5)circle(.1);
\filldraw [black](3.5,4.5)circle(.1);
\filldraw [black](4.5,5.5)circle(.1);
\draw [black](4.5,4.5)circle(.1);
\filldraw [black](4.5,1.5)circle(.1);
\filldraw [black](5.5,6.5)circle(.1);
\filldraw [black](6.5,4.5)circle(.1);

\draw[thick, color=blue] (.5,0)--(.5,2.5)--(7,2.5);
\draw[thick, color=blue] (1.5,0)--(1.5,3.5)--(7,3.5);
\draw[thick, color=blue] (2.5,0)--(2.5,.5)--(7,.5);
\draw[thick, color=blue] (3.5,0)--(3.5,4.5)--(4.5,4.5);
\draw[thick, color=blue] (4.5,0)--(4.5,1.5)--(7,1.5);
\draw[thick, color=blue] (4.5,4.6)--(4.5,5.5)--(7,5.5);
\draw[thick, color=blue] (5.5,0)--(5.5,6.5)--(7,6.5);
\draw[thick, color=blue] (6.5,0)--(6.5,4.5)--(7,4.5);

\node at (-1,4.5){$i$};
\node at (-1,3.5){$i+1$};
\node at (1.5,7.5){$j$};
\node at (3.5,7.5){$\ell$};
\node at (3.5,8.5){$A:$};
\end{tikzpicture} \hspace{.2cm}
\begin{tikzpicture}[x=1.4em,y=1.4em]
\filldraw[color=black, fill=yellow](1,5)rectangle(2,4);
\filldraw[color=black, fill=yellow](4,5)rectangle(5,4);
\draw[step=1,gray, thin] (0,0) grid (7,7);
\draw[color=black, thick](0,0)rectangle(7,7);
\filldraw [black](0.5,2.5)circle(.1);
\filldraw [black](1.5,3.5)circle(.1);
\filldraw [black](2.5,0.5)circle(.1);
\filldraw [black](2.5,4.5)circle(.1);
\draw [black](2.5,3.5)circle(.1);
\filldraw [black](3.5,3.5)circle(.1);
\filldraw [black](4.5,5.5)circle(.1);
\draw [black](4.5,4.5)circle(.1);
\filldraw [black](4.5,1.5)circle(.1);
\filldraw [black](5.5,6.5)circle(.1);
\filldraw [black](6.5,4.5)circle(.1);

\draw[thick, color=blue] (.5,0)--(.5,2.5)--(7,2.5);
\draw[thick, color=blue] (1.5,0)--(1.5,3.5)--(2.4,3.5);
\draw[thick, color=blue] (2.5,0)--(2.5,.5)--(7,.5);
\draw[thick, color=blue] (2.5,3.6)--(2.5,4.5)--(4.4,4.5);
\draw[thick, color=blue] (3.5,0)--(3.5,3.5)--(7,3.5);
\draw[thick, color=blue] (4.5,0)--(4.5,1.5)--(7,1.5);
\draw[thick, color=blue] (4.5,4.6)--(4.5,5.5)--(7,5.5);
\draw[thick, color=blue] (5.5,0)--(5.5,6.5)--(7,6.5);
\draw[thick, color=blue] (6.5,0)--(6.5,4.5)--(7,4.5);

\node at (-1,4.5){$i$};
\node at (-1,3.5){$i+1$};
\node at (1.5,7.5){$j$};
\node at (3.5,7.5){$\ell$};
\node at (3.5,8.5){$B':$};
\end{tikzpicture} \hspace{.2cm} 
\begin{tikzpicture}[x=1.4em,y=1.4em]
\filldraw[color=black, fill=yellow](4,5)rectangle(5,4);
\draw[step=1,gray, thin] (0,0) grid (7,7);
\draw[color=black, thick](0,0)rectangle(7,7);
\filldraw [black](0.5,2.5)circle(.1);
\filldraw [black](1.5,4.5)circle(.1);
\filldraw [black](2.5,0.5)circle(.1);
\filldraw [black](3.5,3.5)circle(.1);
\filldraw [black](4.5,5.5)circle(.1);
\draw [black](4.5,4.5)circle(.1);
\filldraw [black](4.5,1.5)circle(.1);
\filldraw [black](5.5,6.5)circle(.1);
\filldraw [black](6.5,4.5)circle(.1);

\draw[thick, color=blue] (.5,0)--(.5,2.5)--(7,2.5);
\draw[thick, color=blue] (1.5,0)--(1.5,4.5)--(4.4,4.5);
\draw[thick, color=blue] (2.5,0)--(2.5,.5)--(7,.5);
\draw[thick, color=blue] (3.5,0)--(3.5,3.5)--(7,3.5);
\draw[thick, color=blue] (4.5,0)--(4.5,1.5)--(7,1.5);
\draw[thick, color=blue] (4.5,4.6)--(4.5,5.5)--(7,5.5);
\draw[thick, color=blue] (5.5,0)--(5.5,6.5)--(7,6.5);
\draw[thick, color=blue] (6.5,0)--(6.5,4.5)--(7,4.5);

\node at (-1,4.5){$i$};
\node at (-1,3.5){$i+1$};
\node at (1.5,7.5){$j$};
\node at (3.5,7.5){$\ell$};
\node at (3.5,8.5){$B:$};
\end{tikzpicture} 
\end{centering}\qedhere
\]
   \end{example}

\begin{remark}
Anticipating the discussion of pipe dream combinatorics in \cref{subsection:bpdweakorder}, or for the reader already familiar with pipe dream combinatorics, we remark that the entry-wise constructions of \Cref{prop:toggle-one-block-in-weak-order} may be replaced by a description in terms of breaking and bending pipes.  For example, one may obtain $B'$ from $A$ by replacing the crossing tile in position $(4,4)$ by a bumping tile and then performing an undroop on the upward-facing elbow within the bumping tile. Restricting to rows $3$ and $4$ and columns $2$, $3$, and $4$, this looks like \[
\raisebox{1.1em}{$A_{[3,4],[2,4]}:$ }
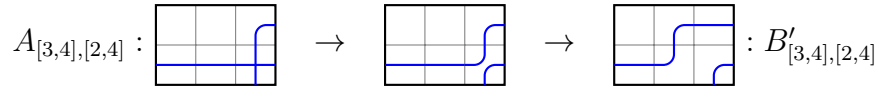
	\begin{tikzpicture}[x=1.25em,y=1.25em]
	\draw[step=1,gray, thin] (0,0) grid (3,2);
	\draw[color=black, thick](0,0)rectangle(3,2);
 \draw[thick,color=blue] (0,.5)--(3,.5);
  \draw[thick,color=blue] (2.5,0)--(2.5,1);
      \draw[thick,rounded corners,color=blue] (2.5,1)--(2.5,1.5)--(3,1.5);
	\end{tikzpicture}
    \hspace{.5cm}
	\raisebox{1.1em}{$\rightarrow$}
	\hspace{.5cm}
	\begin{tikzpicture}[x=1.25em,y=1.25em]
	\draw[step=1,gray, thin] (0,0) grid (3,2);
	\draw[color=black, thick](0,0)rectangle(3,2);
 \draw[thick,color=blue] (0,.5)--(2,.5);
   \draw[thick,rounded corners,color=blue] (2.5,1)--(2.5,1.5)--(3,1.5);
   \draw[thick,rounded corners,color=blue] (2,.5)--(2.5,.5)--(2.5,1);
    \draw[thick,rounded corners,color=blue] (2.5,0)--(2.5,.5)--(3,0.5);
	\end{tikzpicture}
    \hspace{.5cm}
	\raisebox{1.1em}{$\rightarrow$}
	\hspace{.5cm}
	\begin{tikzpicture}[x=1.25em,y=1.25em]
	\draw[step=1,gray, thin] (0,0) grid (3,2);
	\draw[color=black, thick](0,0)rectangle(3,2);
 \draw[thick,color=blue] (0,.5)--(1,.5);
   \draw[thick,color=blue] (2,1.5)--(3,1.5);
      \draw[thick,rounded corners,color=blue] (1.5,1)--(1.5,0.5)--(1,0.5);
    \draw[thick,rounded corners,color=blue] (1.5,1)--(1.5,1.5)--(2,1.5);
        \draw[thick,rounded corners,color=blue] (2.5,0)--(2.5,.5)--(3,0.5);
	\end{tikzpicture}
    \raisebox{1.1em}{ $:B'_{[3,4],[2,4]}$}
\]
\end{remark}

\begin{remark}
\label{rmk:Daoji's Question}
To each $A \in \asm(n)$, one may associate a permutation $w_A$, called the \emph{key} of $A$ in \cite{Las-Ice}.  There is an equivalent characterization via bumpless pipe dreams (see \cite{Wei21}). If $A \in S_n$, then $w_A = A$. 

We thank Daoji Huang for asking if $w_{\pi_i(A)} = \pi_i(w_A)$. If $A$ is the non-permutation element of $\asm(3)$ and $i = 2$, then $\pi_2(A) = 213$, so $w_{\pi_2(A)} = 213$. Meanwhile, $w_A = 132$, so $\pi_2(w_A) = 123 \neq 213$. This example shows that the answer is in general negative.
\end{remark}

\subsection{Computing $\pi_i(A)$ via $m_{A^T}$}
\label{sect:computing-pi_i(A)-with-column-weak-order}

Given $A\in \asm(n)$ and $i \in [n-1]$, let $\pi_i^C(A)=\pi_i(A^T)^T$. We will define column weak order from the operators $\pi^C_i$ just as weak order is defined from the $\pi_i$:  We say $A$ has a \newword{descent in column $i$} if $\pi_i^C(A)\neq A$ and an \newword{ascent in column $i$} otherwise. Define the \newword{column weak order} $\preceq_C$ to be the transitive closure of the covering relations $\pi_i^C(A)\prec_C A$, where $A$ has a descent in column $i$. For $w \in S_n$, $\pi_i^C(w) = s_iw$ if $s_iw<w$ and $\pi_i^C(w) = w$ otherwise, i.e., column weak order on $\asm(n)$ extends left weak order on $S_n$. For $w \in S_n$, we will use left ascent or left descent to mean an ascent or descent with respect to $\pi_i^C(w)$. The fundamental facts about weak order we have proved up to this point translate naturally to column weak order. 

In this section, we give a new way to compute column weak order covers on ASMs directly in terms of monotone triangles.  Column weak order admits a simple description in terms of monotone triangles for the transposes of the corresponding ASMs.  Later, in \cref{subsection:bpdweakorder}, we will apply this description of weak order in order to interpret weak order covers in terms of bumpless pipe dreams (\cref{def:bpd}).

\begin{lemma}
\label{lemma:monotone_transpose_weak_order_permutations}
    Let $w\in S_n$ and $i\in[n-1]$.  Suppose we obtain the array $m$ from $m_w$ as follows: Whenever $i+1$ appears in a row of which $i$ is not an entry, replace $i+1$ with $i$.  Then $m=m_{\pi_i^C(w)}$.
\end{lemma}
\begin{proof}

Let $a,b\in [n]$ such that $w(a)=i$ and $w(b)=i+1$.
By the definition of $m_w$, we have that the first row in which $i$ appears in $m_w$ is row $a$.  Because $m_w$ is a key, $i$ appears in every row thereafter.  Similarly, the first row in which $i+1$ appears is row $b$, and $i+1$ remains thereafter.

\noindent{ \bf Case 1:} $a<b$. In this case, $w$ has a left ascent at $i$ and so $\pi_i^C(w)=w$.  Because $a<b$, any row in $m_w$ that contains $i+1$ will also contain $i$, and so $m=m_w$, as desired.

\noindent{ \bf Case 2:} $a>b$.
    In this case, $s_i w<w$.  Because $w\in S_n$, $\pi_i^C(w)=s_i w$.
Because $a>b$, the only rows in $m_w$ in which $i+1$ appears without $i$ are rows $b, b+1,\ldots, a-1$.  To obtain $m$, we change each of these $i+1$'s to $i$'s.  Because $s_i w$ only differs from $w$ by exchanging the values $i$ and $i+1$, we have that $m=m_{s_iw}$, as desired.
\end{proof}

\begin{proposition}
\label{prop:monotone_transpose_weak_order}
Let $A\in \asm(n)$ and $i\in[n-1]$.  Suppose we obtain the array $m$ from $m_A$ as follows: Whenever $i+1$ appears in a row without an $i$, replace $i+1$ with $i$.  Then $m=m_{\pi_i^C(A)}$.
\end{proposition}

\begin{proof}
We may write
\[
A = \bigvee_{j=1}^t u_j,
\]
where each $u_j \in S_n$. 

By \cite[Proposition 3.4]{EKW-main}, the operator $\pi_i^C$ preserves joins, and hence
\[
\pi_i^C(A) = \bigvee_{j=1}^t \pi_i^C(u_j).
\]
Consequently, for all $1 \le b \le a \le n$,
\[
m_A(a,b) = \max\{ m_{u_j}(a,b) : j \in [t] \},\]
and
\[m_{\pi_i^C(A)}(a,b)
= \max\{ m_{\pi_i^C(u_j)}(a,b) : j \in [t] \}.
\]

Fix $(a,b)$ with $1 \le b \le a \le n$.

\noindent\textbf{Case 1:} $m_A(a,b) = k \neq i+1$.

In this case, there exists some $h \in [t]$ such that $m_{u_h}(a,b) = k$, and for all $j \in [t]$ we have
$m_{u_j}(a,b) \le k$.
By \cref{lemma:monotone_transpose_weak_order_permutations}, since $k \neq i+1$, it follows that
\[
m_{\pi_i^C(u_h)}(a,b) = k
\quad\text{and}\quad
m_{\pi_i^C(u_j)}(a,b) \le k
\text{ for all } j \in [t].
\]
Therefore,
\[
m_{\pi_i^C(A)}(a,b)
= \max\{ m_{\pi_i^C(u_j)}(a,b) : j \in [t] \}
= k.
\]

Because $m$ is produced from $m_A$ by replacing only values $i+1$ and we assume that $k \neq i+1$, we have $m(a,b) = m_A(a,b) = k$. So $m(a,b) = m_{\pi_i^C(A)}(a,b)$, as desired.

\noindent\textbf{Case 2:} $m_A(a,b) = i+1$.

In this case, there exists some $h \in [t]$ such that $m_{u_h}(a,b) = i+1$, and for all $j \in [t]$ we have
$m_{u_j}(a,b) \le i+1$.

\noindent\textbf{Subcase 1:}
For all $j \in [t]$ such that $m_{u_j}(a,b) = i+1$, there is no $i$ in row $a$ of $m_{u_j}$.

In this subcase, by \cref{lemma:monotone_transpose_weak_order_permutations}, for each such $j$ we have
\[
m_{\pi_i^C(u_j)}(a,b) = i.
\]
If $m_{u_j}(a,b) \neq i+1$, then
\[
m_{\pi_i^C(u_j)}(a,b)
= m_{u_j}(a,b)
\le i.
\]
Therefore,
\[
m_{\pi_i^C(A)}(a,b)
= \max\{ m_{\pi_i^C(u_j)}(a,b) : j \in [t] \}
= i.
\]

Because there is no $i$ in row $a$ of any $m_{u_j}$, there is no $i$ in row $a$ of $m_A$. Hence, $m(a,b) = i$ also, by the construction of $m$.

\noindent\textbf{Subcase 2:}
There exists some $j \in [t]$ such that
$m_{u_j}(a,b) = i+1$ and $m_{u_j}(a,b-1) = i$.

In this subcase, $m_{\pi_i^C(u_j)}(a,b) = i+1$, which implies
$m_{\pi_i^C(A)}(a,b) \geq i+1$. And $m_A(a,b) = i+1$ implies $m_{\pi_i^C(A)}(a,b) \leq i+1$. By construction of $m$ and the assumption $m_A(a,b) = i+1$, we know $m(a,b) \in \{i, i+1\}$ and that $m(a,b) = i+1$ if and only if $m_A$ has an $i$ in row $a$. Because $m_{u_j}(a,b-1) = i$, we have $m_A(a,b-1) = i$ unless there is some $j' \in [t]$ with $m_{u_{j'}}(a,b-1) >i$. But monotone triangles strictly increase along rows, so $m_{u_{j'}}(a,b-1) >i$ implies $m_{u_{j'}}(a,b) >i+1$, which requires $m_A(a,b)>i+1$, a contradiction. Hence $m_{\pi_i^C(A)}(a,b) = i+1 = m(a,b)$.
\end{proof}

\begin{remark}
    We can readily use \cref{prop:monotone_transpose_weak_order} to compute $\pi_i(A)$ by applying the statement to the monotone triangle for $A^T$, thus obtaining the monotone triangle for $\pi_i^C(A^T)=(\pi_i(A))^T$. 
\end{remark}

\begin{example}
    Let $A=\begin{pmatrix}
0  &0 & 1 & 0\\
 0 & 1& -1 & 1\\
 1 & 0 & 0 & 0\\
 0 & 0 & 1 & 0
    \end{pmatrix}$.  Then \[m_{A^T}=
\begin{array}{cccc}
3\\
2&3\\
1&3&4\\
1&2&3&4
\end{array}.
\]
By applying \cref{prop:monotone_transpose_weak_order}, we see that 
\[m_{\pi_2^C(A^T)}=
\begin{array}{cccc}
2\\
2&3\\
1&2&4\\
1&2&3&4
\end{array}
\]
and thus
$\pi_2(A)=\begin{pmatrix}
0  &0 & 1 & 0\\
 1 & 0& 0 & 0\\
 0 & 1 & -1 & 1\\
 0 & 0 & 1 & 0
    \end{pmatrix}$.
\end{example}

\subsection{Computing $\pi_i(A)$ via bumpless pipe dreams}
\label{subsection:bpdweakorder}

Consider the six tiles pictured below.
\begin{equation*}
    \begin{tikzpicture}[x=1.25em,y=1.25em]
\draw[step=1,gray, thin] (0,0) grid (1, -1);
\draw[color=black, thick] (0,0) rectangle (1, -1);
\draw[thick,rounded corners,color=blue](1/2, -1)--(1/2, -1/2)--(1, -1/2);
\end{tikzpicture}
\hspace{1em}
\begin{tikzpicture}[x=1.25em,y=1.25em]
\draw[step=1,gray, thin] (0,0) grid (1, -1);
\draw[color=black, thick] (0,0) rectangle (1, -1);
\draw[thick,rounded corners,color=blue](1/2, 0)--(1/2, -1/2)--(0, -1/2);
\end{tikzpicture}
\hspace{1em}
\begin{tikzpicture}[x=1.25em,y=1.25em]
\draw[step=1,gray, thin] (0,0) grid (1, -1);
\draw[color=black, thick] (0,0) rectangle (1, -1);
\draw[thick,rounded corners,color=blue](0, -1/2)--(1, -1/2);
\draw[thick,rounded corners,color=blue](1/2, 0)--(1/2, -1);
\end{tikzpicture}
\hspace{1em}
\begin{tikzpicture}[x=1.25em,y=1.25em]
\draw[step=1,gray, thin] (0,0) grid (1, -1);
\draw[color=black, thick] (0,0) rectangle (1, -1);
\end{tikzpicture}
\hspace{1em}
\begin{tikzpicture}[x=1.25em,y=1.25em]
\draw[step=1,gray, thin] (0,0) grid (1, -1);
\draw[color=black, thick] (0,0) rectangle (1, -1);
\draw[thick,rounded corners,color=blue](0, -1/2)--(1, -1/2);
\end{tikzpicture}
\hspace{1em}
\begin{tikzpicture}[x=1.25em,y=1.25em]
\draw[step=1,gray, thin] (0,0) grid (1, -1);
\draw[color=black, thick] (0,0) rectangle (1, -1);
\draw[thick,rounded corners,color=blue](1/2, 0)--(1/2, -1);
\end{tikzpicture}
\end{equation*}
We may tile the $n\times n$ grid with these tiles and consider the result to be a network of pipes.
\begin{definition}\label{def:bpd} A \newword{bumpless pipe dream} is a tiling such that
\begin{enumerate}
    \item there are $n$ pipes total, and
    \item each pipe starts at the bottom of the grid and exits at the right edge of the grid.
\end{enumerate}
\end{definition}
Write $\bpd(n)$ for the set of bumpless pipe dreams in an $n \times n$ grid.
There is a simple bijection from $\bpd(n)$ to $\asm(n)$ \cite[Section 4]{Beh08}.  The bumpless pipe dream $\mathcal B\in \bpd(n)$ maps to the matrix $A\in \asm(n)$ which is defined by
\[A_{i,j}=\begin{cases}
 1   & \text{if cell } (i,j) \text{ in } \mathcal B \text{ is } \begin{tikzpicture}[x=1em,y=1em]
\draw[step=1,gray, thin] (0,0) grid (1, -1);
\draw[color=black, thick] (0,0) rectangle (1, -1);
\draw[thick,rounded corners,color=blue](1/2, -1)--(1/2, -1/2)--(1, -1/2);
\end{tikzpicture} \\
 -1   &\text{if cell } (i,j) \text{ in } \mathcal B \text{ is } 
    \begin{tikzpicture}[x=1em,y=1em]
\draw[step=1,gray, thin] (0,0) grid (1, -1);
\draw[color=black, thick] (0,0) rectangle (1, -1);
\draw[thick,rounded corners,color=blue](1/2, 0)--(1/2, -1/2)--(0, -1/2);
\end{tikzpicture}
    \\
  0  & \text{otherwise}.\\
\end{cases}\]

Given $A\in \asm(n)$, let $\asmtobpd(A)$ be its corresponding bumpless pipe dream.  
Note that the BPD associated to an ASM $A$ is essentially the visualization of its Rothe diagram, though we draw the corners rounded, as is often the style when working with pipe dreams.

\begin{lemma}
\label{lemma:bpdcornersumdictionary}
    Let $i,j\in [n]$ and write $\rk_A(i,j)=k$.  The restriction of $\rk_A$ to $\{i-1,i\}\times \{j-1,j\}$ is determined by the type of tile in position $(i,j)$ of $\asmtobpd(A)$ as pictured below.
    \begin{center}
    \begin{tabular}{|c|c|c|c|c|c|}\hline
\begin{tikzpicture}[x=1.25em,y=1.25em] 
\draw[step=1,gray, thin] (0,0) grid (1, -1);
\draw[color=black, thick] (0,0) rectangle (1, -1);
\draw[thick,rounded corners,color=blue](1/2, -1)--(1/2, -1/2)--(1, -1/2);
\end{tikzpicture}  
&  \begin{tikzpicture}[x=1.25em,y=1.25em]
\draw[step=1,gray, thin] (0,0) grid (1, -1);
\draw[color=black, thick] (0,0) rectangle (1, -1);
\draw[thick,rounded corners,color=blue](1/2, 0)--(1/2, -1/2)--(0, -1/2);
\end{tikzpicture}
& \begin{tikzpicture}[x=1.25em,y=1.25em]
\draw[step=1,gray, thin] (0,0) grid (1, -1);
\draw[color=black, thick] (0,0) rectangle (1, -1);
\draw[thick,rounded corners,color=blue](0, -1/2)--(1, -1/2);
\draw[thick,rounded corners,color=blue](1/2, 0)--(1/2, -1);
\end{tikzpicture}
& \begin{tikzpicture}[x=1.25em,y=1.25em]
\draw[step=1,gray, thin] (0,0) grid (1, -1);
\draw[color=black, thick] (0,0) rectangle (1, -1);
\end{tikzpicture}
&\begin{tikzpicture}[x=1.25em,y=1.25em]
\draw[step=1,gray, thin] (0,0) grid (1, -1);
\draw[color=black, thick] (0,0) rectangle (1, -1);
\draw[thick,rounded corners,color=blue](0, -1/2)--(1, -1/2);
\end{tikzpicture} 
&\begin{tikzpicture}[x=1.25em,y=1.25em]
\draw[step=1,gray, thin] (0,0) grid (1, -1);
\draw[color=black, thick] (0,0) rectangle (1, -1);
\draw[thick,rounded corners,color=blue](1/2, 0)--(1/2, -1);
\end{tikzpicture} 
\\ \hline
      $\begin{pmatrix}
          k-1&k-1\\
          k-1&k
      \end{pmatrix}   $
      & 
      $\begin{pmatrix}
          k-1&k\\
          k& k
      \end{pmatrix}$
         & 
         $\begin{pmatrix}
          k-2&k-1\\
          k-1& k
      \end{pmatrix}$
         &
        $ \begin{pmatrix}
          k&k\\
          k& k
      \end{pmatrix}$
         & 
       $  \begin{pmatrix}
          k-1&k-1\\
          k& k
      \end{pmatrix}$
         & 
        $ \begin{pmatrix}
          k-1&k\\
          k-1& k
      \end{pmatrix}$
         \\ \hline
    \end{tabular}
    \end{center}
\end{lemma}
\begin{proof}
    This follows immediately from \cite[Lemma 3.3]{Wei21} and \cite[Lemma 3.5]{Wei21}.
\end{proof}

We have previously seen bijections from $\bpd(n)$ to $\asm(n)$ and from $\asm(n)$ to $\monotone(n)$. The following lemma gives a direct construction of the composite. 

\begin{lemma}
\label{lemma:bpdtomonotone}
    Let $\mathcal{B} \in \bpd(n)$ and $A = \asmtobpd^{-1}(\mathcal{B})$.
    \begin{enumerate}
        \item The $i$th row of $m_A$, i.e., the set $\{m_A(i,1),m_A(i,2),\ldots, m_A(i,i)\}$ is the set of column indices $j$ so that a vertical pipe in column $j$ of $\mathcal B$ passes from row $i$ to row $i+1$.
        \item The $j$th row of $m_{A^T}$, i.e., the set $\{m_{A^T}(j,1),m_{A^T}(j,2),\ldots, m_{A^T}(j,j)\}$ is the set of row indices $i$ so that a horizontal pipe in row $i$ of $\mathcal B$ passes from column $j$ to $j+1$.
    \end{enumerate}
\end{lemma}
\begin{proof}
   Let $A\in \asm(n)$ and $\mathcal B=\asmtobpd(A)$.  (1) follows by observing that $\widetilde{A}$ (as in \cref{section:monotone_triangles}) has the following property: 
   $\widetilde{A}_{i,j}=1$ if and only if a vertical pipe in $\mathcal B$ in column $j$ exits row $i$ and passes into row $i+1$.  One can verify this directly from \cref{lemma:bpdcornersumdictionary} along with the observation that $\widetilde{A}_{i,j}=\rk_A(i,j)-\rk_A(i,j-1)$.

   The second claim follows by transpose symmetry.
\end{proof}

\begin{example}
    Let $A$ be the ASM from \cref{ex: asm to monotone}.  Its corresponding BPD $\mathcal{B}$ %(annotated with some markers $*$) 
    and the monotone triangle $m_A$ are pictured below.
    \[
    \raisebox{-3.5em}
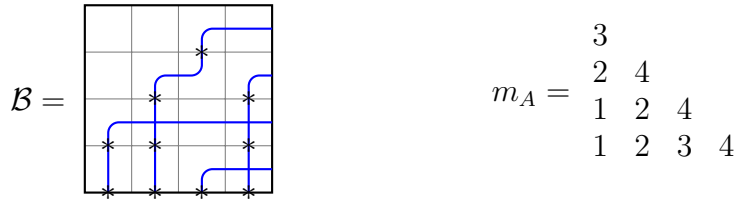
{\begin{tikzpicture}[x=1.5em,y=1.5em]
\draw[step=1,gray, thin] (0,0) grid (4,4);
\draw[color=black, thick](0,0)rectangle(4,4);
\draw[thick, rounded corners, color=blue] (.5,0)--(.5,1.5)--(4,1.5);
\draw[thick, rounded corners, color=blue] (1.5,0)--(1.5,2.5)--(2.5,2.5)--(2.5,3.5)--(4,3.5);
\draw[thick, rounded corners, color=blue] (2.5,0)--(2.5,.5)--(4,.5);
\draw[thick, rounded corners, color=blue] (3.5,0)--(3.5,2.5)--(4,2.5);
\node at (.5,0) {$*$};
\node at (1.5,0) {$*$};
\node at (2.5,0) {$*$};
\node at (3.5,0) {$*$};
\node at (.5,1) {$*$};
\node at (1.5,1) {$*$};
\node at (3.5,1) {$*$};
\node at (3.5,2) {$*$};
\node at (1.5,2) {$*$};
\node at (2.5,3) {$*$};
\node at (-1,2) {$\mathcal{B} = $};
\end{tikzpicture}}
\hspace{7em}
m_A =
\begin{array}{cccc}
3\\
2&4\\
1&2&4\\
1&2&3&4
\end{array}
\]
We have marked $\mathcal{B}$ with a $*$ in each place where a vertical pipe passes from one row to the next. The reader may verify that the columns of these marked crossings agree with the entries of $m_A$.
\end{example}

By combining \cref{lemma:bpdtomonotone} and \cref{prop:monotone_transpose_weak_order}, we obtain a BPD description of weak order, which we will explain below.

Fix $i\in[n-1]$ and a BPD $\mathcal B$.  Call a pipe in row $i+1$ that crosses from some column $j$ to column $j+1$ a \newword{liftable column crossing for $i$} if there is no pipe in row $i$ that crosses from column $j$ to column $j+1$.  Write $\lc_i(\mathcal B)$ for the set of column indices of the liftable column crossings for $i$ in $\mathcal B$.

We may express $\lc_i(\mathcal B)$ as a disjoint union of maximal size intervals, i.e.,
\[\lc_i(\mathcal B)=[a_1,b_1]\sqcup [a_2,b_2]\sqcup \cdots \sqcup [a_k,b_k]\]
with $a_h\leq b_h$ for all $h\in[k]$ and $b_h+1<a_{h+1}$ for all $h\in [k-1]$.  

Call the intervals $[a_h,b_h]$ \newword{liftable $i$-intervals} in $\mathcal B$.
For each $h\in[k]$, make the following local replacements of pipes within the rectangle $[i,i+1]\times[a_h,b_h+1]$ in $\mathcal B$:
\begin{enumerate}
    \item delete the horizontal pipe segment in row $i+1$ from the midpoint of column $a_h$ to the midpoint of column $b_h+1$,
    \item  draw a horizontal pipe segment in row $i$ from the midpoint of column $a_h$ to the midpoint of column $b_h+1$,
    \item delete the vertical pipe segment in column $b_h+1$ from the midpoint of row $i$ to the midpoint of row $i+1$, and
    \item draw a vertical pipe segment in column $a_h$ from the midpoint of row $i$ to the midpoint of row $i+1$.
\end{enumerate}
After doing these replacements, we obtain a new diagram $\mathcal B'$ which we will show is a valid BPD.  Call this map $\psi_i$. Because the intervals $[a_h,b_h+1]$ are disjoint, we may make all of these replacements simultaneously.  See \cref{ex:psi_i} for an example of these replacements.

\begin{proposition}
\label{prop:weakorderbpd}
    The map $\psi_i$ is well-defined.  Moreover, if $A \in \asm(n)$ is such that $\mathcal{B}=\asmtobpd(A)$, then $\psi_i(\mathcal{B})=\asmtobpd(\pi_i(A))$.
\end{proposition}

We postpone the proof of \cref{prop:weakorderbpd} to later in this section.  First, an example.

\begin{example}
\label{ex:psi_i}
    Let $A$ be as in \cref{ex:B'-construction} and $i=3$.  The corresponding BPD $\mathcal B=\asmtobpd(A)$ is pictured below.
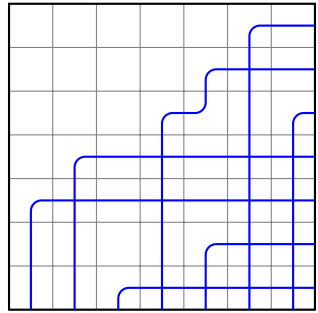
\[\begin{tikzpicture}[x=1.4em,y=1.4em]
\draw[step=1,gray, thin] (0,0) grid (7,7);
\draw[color=black, thick](0,0)rectangle(7,7);

\draw[thick, rounded corners, color=blue] (.5,0)--(.5,2.5)--(7,2.5);
\draw[thick, rounded corners, color=blue] (1.5,0)--(1.5,3.5)--(7,3.5);
\draw[thick, rounded corners, color=blue] (2.5,0)--(2.5,.5)--(7,.5);
\draw[thick, rounded corners, color=blue] (3.5,0)--(3.5,4.5)--(4.5,4.5)--(4.5,5.5)--(7,5.5);
\draw[thick, rounded corners, color=blue] (4.5,0)--(4.5,1.5)--(7,1.5);
\draw[thick, rounded corners, color=blue] (5.5,0)--(5.5,6.5)--(7,6.5);
\draw[thick, rounded corners, color=blue] (6.5,0)--(6.5,4.5)--(7,4.5);
\end{tikzpicture} \]
We have $\lc_i(\mathcal B)=\{2,3,5,6\}=[2,3]\sqcup[5,6]$.  Below, in the picture on the left, we highlight in red the segments of pipes which are deleted from $\mathcal{B}$.  In the next picture to the right, the red pipes are broken at points of intersection, replacing two crossing tiles with bumping tiles. In the third picture, we show the result of what are now commonly called undroop moves (the inverses of the droop moves of \cite{LLS21}) performed on the red segments of pipes.  In the final picture, on the line below, we replace the resultant bumping tile with a crossing tile. The red pipe segments now represent those that are added to obtain $\psi_i(\mathcal{B})$. 

\[\begin{tikzpicture}[x=1.4em,y=1.4em]
\draw[step=1,gray, thin] (0,0) grid (7,7);
\draw[color=black, thick](0,0)rectangle(7,7);
\draw[thick, rounded corners, color=blue] (.5,0)--(.5,2.5)--(7,2.5);
\draw[thick, rounded corners, color=blue] (1.5,0)--(1.5,3.5)--(7,3.5);
\draw[thick, rounded corners, color=blue] (2.5,0)--(2.5,.5)--(7,.5);
\draw[thick, rounded corners, color=blue] (3.5,0)--(3.5,4.5)--(4.5,4.5)--(4.5,5.5)--(7,5.5);
\draw[thick, rounded corners, color=blue] (4.5,0)--(4.5,1.5)--(7,1.5);
\draw[thick, rounded corners, color=blue] (5.5,0)--(5.5,6.5)--(7,6.5);
\draw[thick, rounded corners, color=blue] (6.5,0)--(6.5,4.5)--(7,4.5);
\draw[ultra thick, color=red](1.6,3.5)--(3.5,3.5)--(3.5,4.4);
\draw[ultra thick, color=red](4.5,3.5)--(6.5,3.5)--(6.5,4.4);
\end{tikzpicture} \hspace{1em} \raisebox{1.9cm}{$\mapsto$} \hspace{1em}
\begin{tikzpicture}[x=1.4em,y=1.4em]
\draw[step=1,gray, thin] (0,0) grid (7,7);
\draw[color=black, thick](0,0)rectangle(7,7);
\draw[thick, rounded corners, color=blue] (.5,0)--(.5,2.5)--(7,2.5);
\draw[thick, rounded corners, color=blue] (1.5,0)--(1.5,3.5)--(3.5,3.5)--(3.5,4.5)--(4.5,4.5)--(4.5,5.5)--(7,5.5);
\draw[thick, rounded corners, color=blue] (2.5,0)--(2.5,.5)--(7,.5);
\draw[thick, rounded corners, color=blue] (3.5,0)--(3.5,3.5)--(6.5,3.5)--(6.5,4.5)--(7,4.5);
\draw[thick, rounded corners, color=blue] (4.5,0)--(4.5,1.5)--(7,1.5);
\draw[thick, rounded corners, color=blue] (5.5,0)--(5.5,6.5)--(7,6.5);
\draw[thick, rounded corners, color=blue] (6.5,0)--(6.5,3.5)--(7,3.5);
\draw[ultra thick, rounded corners, color=red](1.6,3.5)--(3.5,3.5)--(3.5,4.4);
\draw[ultra thick, rounded corners, color=red](4.5,3.5)--(6.5,3.5)--(6.5,4.4);
\end{tikzpicture}
\hspace{1em} \raisebox{1.9cm}{$\mapsto$} \hspace{1em}
\begin{tikzpicture}[x=1.4em,y=1.4em]
\draw[step=1,gray, thin] (0,0) grid (7,7);
\draw[color=black, thick](0,0)rectangle(7,7);
\draw[thick, rounded corners, color=blue] (.5,0)--(.5,2.5)--(7,2.5);
\draw[thick, rounded corners, color=blue] (1.5,0)--(1.5,4.5)--(4.5,4.5)--(4.5,5.5)--(7,5.5);
\draw[thick, rounded corners, color=blue] (2.5,0)--(2.5,.5)--(7,.5);
\draw[thick, rounded corners, color=blue] (3.5,0)--(3.5,3.5)--(4.5,3.5)--(4.5,4.5)--(7,4.5);
\draw[thick, rounded corners, color=blue] (4.5,0)--(4.5,1.5)--(7,1.5);
\draw[thick, rounded corners, color=blue] (5.5,0)--(5.5,6.5)--(7,6.5);
\draw[thick, rounded corners, color=blue] (6.5,0)--(6.5,3.5)--(7,3.5);
\draw[ultra thick, rounded corners, color=red](1.5,3.5)--(1.5,4.5)--(3.5,4.5);
\draw[ultra thick, rounded corners, color=red](4.5,3.6)--(4.5,4.5)--(6.5,4.5);
\end{tikzpicture}
\]
This gives us $\psi_i(\mathcal{B})$:
\[\begin{tikzpicture}[x=1.4em,y=1.4em]
\draw[step=1,gray, thin] (0,0) grid (7,7);
\draw[color=black, thick](0,0)rectangle(7,7);
\draw[thick, rounded corners, color=blue] (.5,0)--(.5,2.5)--(7,2.5);
\draw[thick, rounded corners, color=blue] (1.5,0)--(1.5,4.5)--(7,4.5);
\draw[thick, rounded corners, color=blue] (2.5,0)--(2.5,.5)--(7,.5);
\draw[thick, rounded corners, color=blue] (3.5,0)--(3.5,3.5)--(4.5,3.5)--(4.5,5.5)--(7,5.5);
\draw[thick, rounded corners, color=blue] (4.5,0)--(4.5,1.5)--(7,1.5);
\draw[thick, rounded corners, color=blue] (5.5,0)--(5.5,6.5)--(7,6.5);
\draw[thick, rounded corners, color=blue] (6.5,0)--(6.5,3.5)--(7,3.5);
\draw[ultra thick,  rounded corners, color=red](1.5,3.5)--(1.5,4.5)--(3.5,4.5);
\draw[ultra thick,  color=red](4.5,3.6)--(4.5,4.5)--(6.5,4.5);
\end{tikzpicture}\]
By \cref{prop:weakorderbpd}, $\psi_i(\mathcal{B})=\asmtobpd(\pi_3(A))$.
\end{example}

\begin{lemma}
\label{lemma:bpdreplacements}
Let $\mathcal B$ be a BPD.  Suppose $[a,b]$ is a liftable $i$-interval in $\mathcal B$ and let $\mathcal B'=\psi_i(\mathcal B)$.

\begin{enumerate}

\item The replacement $\mathcal B\mapsto \mathcal B'$, restricted to $[i,i+1]\times[a,a]$, must be one of the following:
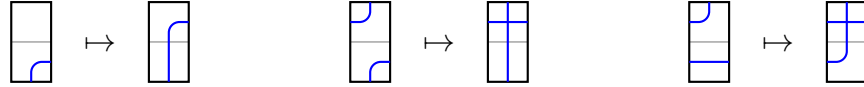
\[
\begin{tikzpicture}[x=1.25em,y=1.25em]
	\draw[step=1,gray, thin] (0,0) grid (1,2);
	\draw[thick] (0,0)--(1,0)--(1,2)--(0,2)--cycle;
	\draw[thick,rounded corners,color=blue] (.5,0)--(.5,.5)--(1,.5);
\end{tikzpicture}
\quad
\raisebox{1em}{$\mapsto$}
\quad
\begin{tikzpicture}[x=1.25em,y=1.25em]
	\draw[step=1,gray, thin] (0,0) grid (1,2);
	\draw[thick] (0,0)--(1,0)--(1,2)--(0,2)--cycle;
	\draw[thick,rounded corners,color=blue] (.5,0)--(.5,1.5)--(1,1.5);
\end{tikzpicture}
\hspace{5em}
\begin{tikzpicture}[x=1.25em,y=1.25em]
	\draw[step=1,gray, thin] (0,0) grid (1,2);
	\draw[thick] (0,0)--(1,0)--(1,2)--(0,2)--cycle;
	\draw[thick,rounded corners,color=blue] (0,1.5)--(.5,1.5)--(.5,2);
	\draw[thick,rounded corners,color=blue] (.5,0)--(.5,.5)--(1,.5);
\end{tikzpicture}
\quad
\raisebox{1em}{$\mapsto$}
\quad
\begin{tikzpicture}[x=1.25em,y=1.25em]
	\draw[step=1,gray, thin] (0,0) grid (1,2);
	\draw[thick] (0,0)--(1,0)--(1,2)--(0,2)--cycle;
	\draw[thick,rounded corners,color=blue] (0,1.5)--(1,1.5);
	\draw[thick,rounded corners,color=blue] (.5,0)--(.5,2);
\end{tikzpicture}
\hspace{5em}
\begin{tikzpicture}[x=1.25em,y=1.25em]
	\draw[step=1,gray, thin] (0,0) grid (1,2);
	\draw[thick] (0,0)--(1,0)--(1,2)--(0,2)--cycle;
	\draw[thick,rounded corners,color=blue] (0,1.5)--(.5,1.5)--(.5,2);
	\draw[thick,rounded corners,color=blue] (0,.5)--(1,.5);
\end{tikzpicture}
\quad
\raisebox{1em}{$\mapsto$}
\quad
\begin{tikzpicture}[x=1.25em,y=1.25em]
	\draw[step=1,gray, thin] (0,0) grid (1,2);
	\draw[thick] (0,0)--(1,0)--(1,2)--(0,2)--cycle;
	\draw[thick,rounded corners,color=blue] (0,1.5)--(1,1.5);
	\draw[thick,rounded corners,color=blue] (0,.5)--(.5,.5)--(.5,2);
\end{tikzpicture}
\]

   \item If $a<j<b+1$, then the replacement $\mathcal B\mapsto \mathcal B'$, restricted to  $[i,i+1]\times[j,j]$, must be one of the following configurations:
    
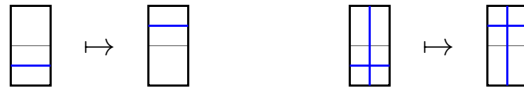
\[
    \begin{tikzpicture}[x=1.25em,y=1.25em]
	\draw[step=1,gray, thin] (0,0) grid (1,2);
	\draw[color=black, thick](0,0)rectangle(1,2);
 	\draw[thick,rounded corners,color=blue](0,.5)--(1,.5);
	\end{tikzpicture}
           \quad
    \raisebox{1em}{$\mapsto$}
    \quad
    \begin{tikzpicture}[x=1.25em,y=1.25em]
	\draw[step=1,gray, thin] (0,0) grid (1,2);
	\draw[color=black, thick](0,0)rectangle(1,2);
 	\draw[thick,rounded corners,color=blue](0,1.5)--(1,1.5);
	\end{tikzpicture}
    \hspace{5em}
    \begin{tikzpicture}[x=1.25em,y=1.25em]
	\draw[step=1,gray, thin] (0,0) grid (1,2);
	\draw[color=black, thick](0,0)rectangle(1,2);
 	\draw[thick,rounded corners,color=blue](0,.5)--(1,.5);
    \draw[thick,rounded corners,color=blue](.5,0)--(.5,2);
	\end{tikzpicture}
           \quad
    \raisebox{1em}{$\mapsto$}
    \quad
     \begin{tikzpicture}[x=1.25em,y=1.25em]
	\draw[step=1,gray, thin] (0,0) grid (1,2);
	\draw[color=black, thick](0,0)rectangle(1,2);
 	\draw[thick,rounded corners,color=blue](0,1.5)--(1,1.5);
    \draw[thick,rounded corners,color=blue](.5,0)--(.5,2);
	\end{tikzpicture}
    \]

\item The replacement $\mathcal B\mapsto \mathcal B'$, restricted to  $[i,i+1]\times [b+1,b+1]$, must be one of the following configurations:
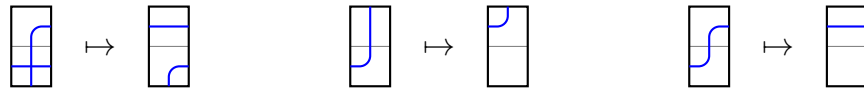
\[
\begin{tikzpicture}[x=1.25em,y=1.25em]
	\draw[step=1,gray, thin] (0,0) grid (1,2);
	\draw[thick] (0,0)--(1,0)--(1,2)--(0,2)--cycle;
	\draw[thick,rounded corners,color=blue] (.5,0)--(.5,1.5)--(1,1.5);
	\draw[thick,rounded corners,color=blue] (0,.5)--(1,.5);
\end{tikzpicture}
\quad
\raisebox{1em}{$\mapsto$}
\quad
\begin{tikzpicture}[x=1.25em,y=1.25em]
	\draw[step=1,gray, thin] (0,0) grid (1,2);
	\draw[thick] (0,0)--(1,0)--(1,2)--(0,2)--cycle;
	\draw[thick,rounded corners,color=blue] (0,1.5)--(1,1.5);
	\draw[thick,rounded corners,color=blue] (.5,0)--(.5,.5)--(1,.5);
\end{tikzpicture}
\hspace{5em}
\begin{tikzpicture}[x=1.25em,y=1.25em]
	\draw[step=1,gray, thin] (0,0) grid (1,2);
	\draw[thick] (0,0)--(1,0)--(1,2)--(0,2)--cycle;
	\draw[thick,rounded corners,color=blue] (0,.5)--(.5,.5)--(.5,2);
\end{tikzpicture}
\quad
\raisebox{1em}{$\mapsto$}
\quad
\begin{tikzpicture}[x=1.25em,y=1.25em]
	\draw[step=1,gray, thin] (0,0) grid (1,2);
	\draw[thick] (0,0)--(1,0)--(1,2)--(0,2)--cycle;
	\draw[thick,rounded corners,color=blue] (0,1.5)--(.5,1.5)--(.5,2);
\end{tikzpicture}
\hspace{5em}
\begin{tikzpicture}[x=1.25em,y=1.25em]
	\draw[step=1,gray, thin] (0,0) grid (1,2);
	\draw[thick] (0,0)--(1,0)--(1,2)--(0,2)--cycle;
	\draw[thick,rounded corners,color=blue] (0,.5)--(.5,.5)--(.5,1.5)--(1,1.5);
\end{tikzpicture}
\quad
\raisebox{1em}{$\mapsto$}
\quad
\begin{tikzpicture}[x=1.25em,y=1.25em]
	\draw[step=1,gray, thin] (0,0) grid (1,2);
	\draw[thick] (0,0)--(1,0)--(1,2)--(0,2)--cycle;
	\draw[thick,rounded corners,color=blue] (0,1.5)--(1,1.5);
\end{tikzpicture}
\]

\end{enumerate}
\end{lemma}

\begin{proof}
    \noindent (1) Because column $a$ is the start of a liftable $i$-interval, $\mathcal B$ must have a pipe exiting column $a$ in row $i+1$ but no pipe exiting column $a$ in row $i$.  Furthermore, we cannot simultaneously have a pipe entering column $a$ in row $i+1$ and no pipe entering in row $i$.  Otherwise, column $a-1$ would be a liftable column crossing for $i$ in $\mathcal B$.
    
The three pictured configurations are the only ones satisfying these requirements. The local replacement within $[i,i+1]\times [a,a]$ in all cases is as described in the definition of $\psi_i$.

    \noindent (2) In this case, in row $i+1$ of $\mathcal B$, a pipe must both enter and exit column $j$ and no pipe may enter or exit column $j$ in row $i$.  The two pictured configurations are the only configurations satisfying these requirements. 
 
 The local replacement within $[i,i+1]\times [j,j]$ in both cases is as described in the definition of $\psi_i$.

    \noindent (3) Because column $b$ is the end of a liftable $i$-interval, $\mathcal{B}$ must have no pipe entering column $b+1$ in row $i$, and must have a pipe entering column $b+1$ in row $i+1$.  Furthermore, we cannot simultaneously have a pipe leaving column $b+1$ in row $i+1$ and no pipe leaving in row $i$.  Otherwise, column $b+1$ would also be a liftable column crossing for $i$.

    The three pictured configurations are the only ones satisfying these requirements.  The local replacement within $[i,i+1]\times [b+1,b+1]$ in all cases is as described in the definition of $\psi_i$.
\end{proof}

We now prove the main proposition of this section.
\begin{proof}[{Proof of \cref{prop:weakorderbpd}}]
    That $\psi_i$ is a well-defined map from $\bpd(n)$ to $\bpd(n)$ follows from \cref{lemma:bpdreplacements}.  The equality $\psi_i(\mathcal B)=\asmtobpd(\pi_i(A))$ follows from \cref{lemma:bpdtomonotone} and \cref{prop:monotone_transpose_weak_order}.
\end{proof}

We conclude this section by summarizing some of the conditions that are equivalent to having no essential cell in row $i$.

\begin{proposition}
\label{equivalent-descent-conditions}
    Let $A\in \asm(n)$.  The following are equivalent:
    \begin{enumerate}
        \item Each row of $m_{A^T}$ that contains $i+1$ also contains $i$.
        \item $\pi_i(A)=A$, i.e., $A$ does not have a descent in row $i$.
        \item $A$ does not have an essential cell in row $i$.
        \item $m_A$ does not have an essential point in row $i$.
        \item There exists $C\in \asm(n)$ such that $C\neq A$ and $\pi_i(C)=A$.
        \item The BPD $\asmtobpd(A)$ has no liftable column crossing for $i$.
    \end{enumerate}
\end{proposition}
\begin{proof}
    The equivalence of (1) and (2) follows from \cref{prop:monotone_transpose_weak_order}.

The equivalence of (3) and (4) is \cref{lemma:monotone_essential_set_translation}.

     (3) implies (5) is by \cite[Lemma 6.2]{HR20} and \cref{weak_order_isomorphism}. To see that (5) implies (2), assume the existence of such a $C$. It is immediate from the definition of the operator $\pi_i$ that it is idempotent. So then $A = \pi_i(C) = \pi_i(\pi_i(C)) = \pi_i(A)$. Now (2) implies (3) by \cref{lemma:descentessential} (which indeed is directly the equivalence of (2) and (3)).

    The equivalence of (4) and (6) follows from \cref{prop:weakorderbpd}.
\end{proof}

Consequently, we may give a statement for ASMs that mirrors Hamaker and Reiner's \cite[Lemma 6.2]{HR20} characterization of monotone triangles that are maximal in weak order:

\begin{corollary}
\label{prop:maximal-iff-essential-cell-in-every-row}
    The ASM $A$ is maximal in weak order if and only if $A$ has an essential cell in each row except the last.
\end{corollary}
\begin{proof}
    This is immediate from the equivalence of (3) and (5) in \cref{equivalent-descent-conditions}.
\end{proof}

\section{Understanding $\pi_i^{-1}(\pi_i(A))$}

Through the equivalence of \cref{equivalent-descent-conditions}, via \cite[Lemma 6.2]{HR20}, we know that $\pi_i^{-1}(\pi_i(A))$ consists of two or more matrices for each $i \in [n-1]$ and each $A \in \asm(n)$. We know from \cite[Lemma 2.9]{EKW-main} that this set is a sublattice of $\asm(n)$.  The purpose of this section is to give a method for constructing the complete set of BPDs corresponding to this sublattice.

We start by reinterpreting the covering relations of the ASM lattice in terms of local replacements on BPDs.  These moves are known as \emph{flips} on BPDs in \cite{APP}.

\begin{proposition}
\label{prop:moves-to-get-preimage}
    Given $A,B\in \asm(n)$, $A$ covers $B$ in strong order if and only if $\asmtobpd(A)$ can be obtained by applying one of the following local moves to $\asmtobpd(B)$:

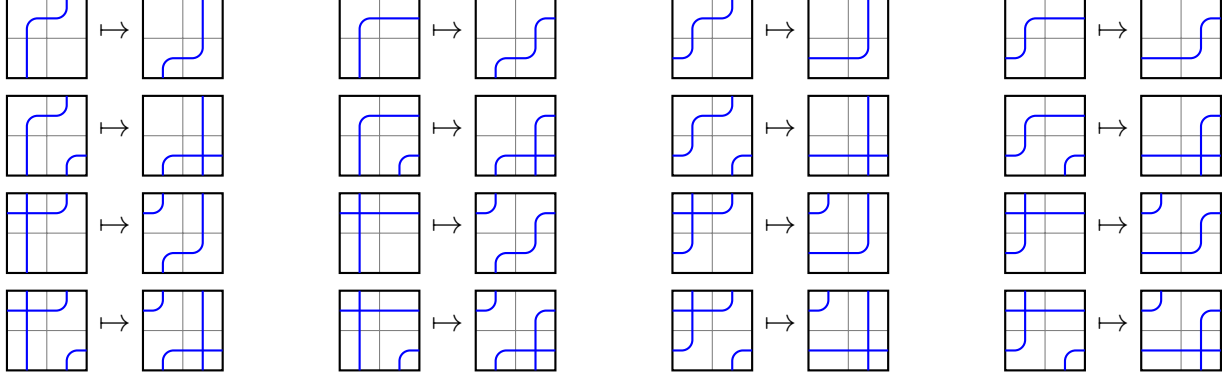
\begin{tikzpicture}[x=1.25em,y=1.25em, baseline=(current bounding box.center)]
\draw[step=1,gray, thin] (0,0) grid (2,2);
\draw[color=black, thick](0,0)rectangle(2,2);
\draw[thick,rounded corners,color=blue] (.5,0)--(.5,1.5)--(1.5,1.5)--(1.5,2);
\end{tikzpicture}
$\mapsto$
\begin{tikzpicture}[x=1.25em,y=1.25em, baseline=(current bounding box.center)]
\draw[step=1,gray, thin] (0,0) grid (2,2);
\draw[color=black, thick](0,0)rectangle(2,2);
\draw[thick,rounded corners,color=blue] (.5,0)--(.5,.5)--(1.5,.5)--(1.5,1.5)--(1.5,2);
\end{tikzpicture} \hfill
\begin{tikzpicture}[x=1.25em,y=1.25em, baseline=(current bounding box.center)]
\draw[step=1,gray, thin] (0,0) grid (2,2);
\draw[color=black, thick](0,0)rectangle(2,2);
\draw[thick,rounded corners,color=blue] (.5,0)--(.5,1.5)--(1.5,1.5)--(2,1.5);
\end{tikzpicture}
$\mapsto$ 
\begin{tikzpicture}[x=1.25em,y=1.25em, baseline=(current bounding box.center)]
\draw[step=1,gray, thin] (0,0) grid (2,2);
\draw[color=black, thick](0,0)rectangle(2,2);
\draw[thick,rounded corners,color=blue] (.5,0)--(.5,.5)--(1.5,.5)--(1.5,1.5)--(2,1.5);
\end{tikzpicture} \hfill
\begin{tikzpicture}[x=1.25em,y=1.25em, baseline=(current bounding box.center)]
\draw[step=1,gray, thin] (0,0) grid (2,2);
\draw[color=black, thick](0,0)rectangle(2,2);
\draw[thick,rounded corners,color=blue] (0,.5)--(.5,.5)--(.5,1.5)--(1.5,1.5)--(1.5,2);
\end{tikzpicture}
$\mapsto$ 
\begin{tikzpicture}[x=1.25em,y=1.25em, baseline=(current bounding box.center)]
\draw[step=1,gray, thin] (0,0) grid (2,2);
\draw[color=black, thick](0,0)rectangle(2,2);
\draw[thick,rounded corners,color=blue] (0,.5)--(.5,.5)--(1.5,.5)--(1.5,1.5)--(1.5,2);
\end{tikzpicture} \hfill
\begin{tikzpicture}[x=1.25em,y=1.25em, baseline=(current bounding box.center)]
\draw[step=1,gray, thin] (0,0) grid (2,2);
\draw[color=black, thick](0,0)rectangle(2,2);
\draw[thick,rounded corners,color=blue] (0,.5)--(.5,.5)--(.5,1.5)--(1.5,1.5)--(2,1.5);
\end{tikzpicture}
 $\mapsto$ 
\begin{tikzpicture}[x=1.25em,y=1.25em, baseline=(current bounding box.center)]
\draw[step=1,gray, thin] (0,0) grid (2,2);
\draw[color=black, thick](0,0)rectangle(2,2);
\draw[thick,rounded corners,color=blue] (0,.5)--(.5,.5)--(1.5,.5)--(1.5,1.5)--(2,1.5);
\end{tikzpicture}

\smallskip

\begin{tikzpicture}[x=1.25em,y=1.25em, baseline=(current bounding box.center)]
\draw[step=1,gray, thin] (0,0) grid (2,2);
\draw[color=black, thick](0,0)rectangle(2,2);
\draw[thick,rounded corners,color=blue] (1.5,0)--(1.5,.5)--(2,.5);
\draw[thick,rounded corners,color=blue] (.5,0)--(.5,1.5)--(1.5,1.5)--(1.5,2);
\end{tikzpicture}
$\mapsto$
\begin{tikzpicture}[x=1.25em,y=1.25em, baseline=(current bounding box.center)]
\draw[step=1,gray, thin] (0,0) grid (2,2);
\draw[color=black, thick](0,0)rectangle(2,2);
\draw[thick,rounded corners,color=blue] (.5,0)--(.5,.5)--(2,.5);
\draw[thick,rounded corners,color=blue] (1.5,0)--(1.5,2);
\end{tikzpicture} \hfill
\begin{tikzpicture}[x=1.25em,y=1.25em, baseline=(current bounding box.center)]
\draw[step=1,gray, thin] (0,0) grid (2,2);
\draw[color=black, thick](0,0)rectangle(2,2);
\draw[thick,rounded corners,color=blue] (1.5,0)--(1.5,.5)--(2,.5);
\draw[thick,rounded corners,color=blue] (.5,0)--(.5,1.5)--(1.5,1.5)--(2,1.5);
\end{tikzpicture}
$\mapsto$
\begin{tikzpicture}[x=1.25em,y=1.25em, baseline=(current bounding box.center)]
\draw[step=1,gray, thin] (0,0) grid (2,2);
\draw[color=black, thick](0,0)rectangle(2,2);
\draw[thick,rounded corners,color=blue] (.5,0)--(.5,.5)--(2,.5);
\draw[thick,rounded corners,color=blue] (1.5,0)--(1.5,1.5)--(2,1.5);
\end{tikzpicture} \hfill
\begin{tikzpicture}[x=1.25em,y=1.25em, baseline=(current bounding box.center)]
\draw[step=1,gray, thin] (0,0) grid (2,2);
\draw[color=black, thick](0,0)rectangle(2,2);
\draw[thick,rounded corners,color=blue] (1.5,0)--(1.5,.5)--(2,.5);
\draw[thick,rounded corners,color=blue] (0,.5)--(.5,.5)--(.5,1.5)--(1.5,1.5)--(1.5,2);
\end{tikzpicture}
$\mapsto$
\begin{tikzpicture}[x=1.25em,y=1.25em, baseline=(current bounding box.center)]
\draw[step=1,gray, thin] (0,0) grid (2,2);
\draw[color=black, thick](0,0)rectangle(2,2);
\draw[thick,rounded corners,color=blue] (0,.5)--(.5,.5)--(2,.5);
\draw[thick,rounded corners,color=blue] (1.5,0)--(1.5,2);
\end{tikzpicture} \hfill
\begin{tikzpicture}[x=1.25em,y=1.25em, baseline=(current bounding box.center)]
\draw[step=1,gray, thin] (0,0) grid (2,2);
\draw[color=black, thick](0,0)rectangle(2,2);
\draw[thick,rounded corners,color=blue] (1.5,0)--(1.5,.5)--(2,.5);
\draw[thick,rounded corners,color=blue] (0,.5)--(.5,.5)--(.5,1.5)--(1.5,1.5)--(2,1.5);
\end{tikzpicture}
$\mapsto$
\begin{tikzpicture}[x=1.25em,y=1.25em, baseline=(current bounding box.center)]
\draw[step=1,gray, thin] (0,0) grid (2,2);
\draw[color=black, thick](0,0)rectangle(2,2);
\draw[thick,rounded corners,color=blue] (0,.5)--(.5,.5)--(2,.5);
\draw[thick,rounded corners,color=blue] (1.5,0)--(1.5,1.5)--(2,1.5);
\end{tikzpicture}

\smallskip

\begin{tikzpicture}[x=1.25em,y=1.25em, baseline=(current bounding box.center)]
\draw[step=1,gray, thin] (0,0) grid (2,2);
\draw[color=black, thick](0,0)rectangle(2,2);
\draw[thick,rounded corners,color=blue] (.5,0)--(.5,2);
\draw[thick,rounded corners,color=blue] (0,1.5)--(1.5,1.5)--(1.5,2);
\end{tikzpicture}
$\mapsto$
\begin{tikzpicture}[x=1.25em,y=1.25em, baseline=(current bounding box.center)]
\draw[step=1,gray, thin] (0,0) grid (2,2);
\draw[color=black, thick](0,0)rectangle(2,2);
\draw[thick,rounded corners,color=blue] (0,1.5)--(.5,1.5)--(.5,2);
\draw[thick,rounded corners,color=blue] (.5,0)--(.5,.5)--(1.5,.5)--(1.5,1.5)--(1.5,2);
\end{tikzpicture} \hfill
\begin{tikzpicture}[x=1.25em,y=1.25em, baseline=(current bounding box.center)]
\draw[step=1,gray, thin] (0,0) grid (2,2);
\draw[color=black, thick](0,0)rectangle(2,2);
\draw[thick,rounded corners,color=blue] (.5,0)--(.5,2);
\draw[thick,rounded corners,color=blue] (0,1.5)--(2,1.5);
\end{tikzpicture}
$\mapsto$
\begin{tikzpicture}[x=1.25em,y=1.25em, baseline=(current bounding box.center)]
\draw[step=1,gray, thin] (0,0) grid (2,2);
\draw[color=black, thick](0,0)rectangle(2,2);
\draw[thick,rounded corners,color=blue] (0,1.5)--(.5,1.5)--(.5,2);
\draw[thick,rounded corners,color=blue] (.5,0)--(.5,.5)--(1.5,.5)--(1.5,1.5)--(2,1.5);
\end{tikzpicture} \hfill
\begin{tikzpicture}[x=1.25em,y=1.25em, baseline=(current bounding box.center)]
\draw[step=1,gray, thin] (0,0) grid (2,2);
\draw[color=black, thick](0,0)rectangle(2,2);
\draw[thick,rounded corners,color=blue] (0,.5)--(.5,.5)--(.5,2);
\draw[thick,rounded corners,color=blue] (0,1.5)--(1.5,1.5)--(1.5,2);
\end{tikzpicture}
$\mapsto$
\begin{tikzpicture}[x=1.25em,y=1.25em, baseline=(current bounding box.center)]
\draw[step=1,gray, thin] (0,0) grid (2,2);
\draw[color=black, thick](0,0)rectangle(2,2);
\draw[thick,rounded corners,color=blue] (0,1.5)--(.5,1.5)--(.5,2);
\draw[thick,rounded corners,color=blue] (0,.5)--(.5,.5)--(1.5,.5)--(1.5,1.5)--(1.5,2);
\end{tikzpicture} \hfill
\begin{tikzpicture}[x=1.25em,y=1.25em, baseline=(current bounding box.center)]
\draw[step=1,gray, thin] (0,0) grid (2,2);
\draw[color=black, thick](0,0)rectangle(2,2);
\draw[thick,rounded corners,color=blue] (0,.5)--(.5,.5)--(.5,2);
\draw[thick,rounded corners,color=blue] (0,1.5)--(2,1.5);
\end{tikzpicture}
$\mapsto$
\begin{tikzpicture}[x=1.25em,y=1.25em, baseline=(current bounding box.center)]
\draw[step=1,gray, thin] (0,0) grid (2,2);
\draw[color=black, thick](0,0)rectangle(2,2);
\draw[thick,rounded corners,color=blue] (0,1.5)--(.5,1.5)--(.5,2);
\draw[thick,rounded corners,color=blue] (0,.5)--(.5,.5)--(1.5,.5)--(1.5,1.5)--(2,1.5);
\end{tikzpicture}

\smallskip

\begin{tikzpicture}[x=1.25em,y=1.25em, baseline=(current bounding box.center)]
\draw[step=1,gray, thin] (0,0) grid (2,2);
\draw[color=black, thick](0,0)rectangle(2,2);
\draw[thick,rounded corners,color=blue] (.5,0)--(.5,2);
\draw[thick,rounded corners,color=blue] (0,1.5)--(1.5,1.5)--(1.5,2);
\draw[thick,rounded corners,color=blue] (1.5,0)--(1.5,.5)--(2,.5);
\end{tikzpicture}
$\mapsto$
\begin{tikzpicture}[x=1.25em,y=1.25em, baseline=(current bounding box.center)]
\draw[step=1,gray, thin] (0,0) grid (2,2);
\draw[color=black, thick](0,0)rectangle(2,2);
\draw[thick,rounded corners,color=blue] (0,1.5)--(.5,1.5)--(.5,2);
\draw[thick,rounded corners,color=blue] (.5,0)--(.5,.5)--(2,.5);
\draw[thick,rounded corners,color=blue] (1.5,0)--(1.5,2);
\end{tikzpicture} \hfill
\begin{tikzpicture}[x=1.25em,y=1.25em, baseline=(current bounding box.center)]
\draw[step=1,gray, thin] (0,0) grid (2,2);
\draw[color=black, thick](0,0)rectangle(2,2);
\draw[thick,rounded corners,color=blue] (.5,0)--(.5,2);
\draw[thick,rounded corners,color=blue] (0,1.5)--(2,1.5);
\draw[thick,rounded corners,color=blue] (1.5,0)--(1.5,.5)--(2,.5);
\end{tikzpicture}
$\mapsto$
\begin{tikzpicture}[x=1.25em,y=1.25em, baseline=(current bounding box.center)]
\draw[step=1,gray, thin] (0,0) grid (2,2);
\draw[color=black, thick](0,0)rectangle(2,2);
\draw[thick,rounded corners,color=blue] (0,1.5)--(.5,1.5)--(.5,2);
\draw[thick,rounded corners,color=blue] (.5,0)--(.5,.5)--(2,.5);
\draw[thick,rounded corners,color=blue] (1.5,0)--(1.5,1.5)--(2,1.5);
\end{tikzpicture} \hfill
\begin{tikzpicture}[x=1.25em,y=1.25em, baseline=(current bounding box.center)]
\draw[step=1,gray, thin] (0,0) grid (2,2);
\draw[color=black, thick](0,0)rectangle(2,2);
\draw[thick,rounded corners,color=blue] (0,.5)--(.5,.5)--(.5,2);
\draw[thick,rounded corners,color=blue] (0,1.5)--(1.5,1.5)--(1.5,2);
\draw[thick,rounded corners,color=blue] (1.5,0)--(1.5,.5)--(2,.5);
\end{tikzpicture}
$\mapsto$
\begin{tikzpicture}[x=1.25em,y=1.25em, baseline=(current bounding box.center)]
\draw[step=1,gray, thin] (0,0) grid (2,2);
\draw[color=black, thick](0,0)rectangle(2,2);
\draw[thick,rounded corners,color=blue] (0,1.5)--(.5,1.5)--(.5,2);
\draw[thick,rounded corners,color=blue] (0,.5)--(.5,.5)--(2,.5);
\draw[thick,rounded corners,color=blue] (1.5,0)--(1.5,2);
\end{tikzpicture} 
\hfill
\begin{tikzpicture}[x=1.25em,y=1.25em, baseline=(current bounding box.center)]
\draw[step=1,gray, thin] (0,0) grid (2,2);
\draw[color=black, thick](0,0)rectangle(2,2);
\draw[thick,rounded corners,color=blue] (0,.5)--(.5,.5)--(.5,2);
\draw[thick,rounded corners,color=blue] (0,1.5)--(2,1.5);
\draw[thick,rounded corners,color=blue] (1.5,0)--(1.5,.5)--(2,.5);
\end{tikzpicture}
$\mapsto$
\begin{tikzpicture}[x=1.25em,y=1.25em, baseline=(current bounding box.center)]
\draw[step=1,gray, thin] (0,0) grid (2,2);
\draw[color=black, thick](0,0)rectangle(2,2);
\draw[thick,rounded corners,color=blue] (0,1.5)--(.5,1.5)--(.5,2);
\draw[thick,rounded corners,color=blue] (0,.5)--(.5,.5)--(2,.5);
\draw[thick,rounded corners,color=blue] (1.5,0)--(1.5,1.5)--(2,1.5);
\end{tikzpicture}
\end{proposition}
\begin{proof}
    We can verify the statement tile by tile using \cref{lemma:bpdcornersumdictionary} and the characterization of covering relations on corner sum matrices, see e.g., \cite[Lemma 2.19]{EKW-main}.
\end{proof}

We now adapt the previous section's construction to produce the maximum element of $\pi_i^{-1}(\pi_i(A))$.  We describe this closely related construction below.

Fix $i\in[n-1]$ and a BPD $\mathcal B$. Call a pipe in row $i$ that crosses from column $j$ to column $j+1$ a \newword{droppable column crossing} for $i$ if there is no pipe in row $i+1$ that crosses from column $j$ to column $j+1$. Write $\dc_i(\mathcal B)$ for the set of column indices of such droppable column crossings for $i$ in $\mathcal B$.

We may express $\dc_i(\mathcal B)$ as a disjoint union of maximal size intervals, i.e.,
\[\dc_i(\mathcal B)=[a_1,b_1]\sqcup [a_2,b_2]\sqcup \cdots \sqcup [a_k,b_k]\]
with $a_h\leq b_h$ for all $h\in[k]$ and $b_h+1<a_{h+1}$ for all $h\in [k-1]$. Call the intervals $[a_h,b_h]$ \newword{droppable $i$-intervals} in $\mathcal B$.

For each $h\in[k]$, make the following local replacements of pipes within the rectangle $[i,i+1]\times[a_h,b_h+1]$ in $\mathcal B$:
\begin{enumerate}
    \item delete the horizontal pipe segment in row $i$ from the midpoint of column $a_h$ to the midpoint of column $b_h+1$,
    \item draw a horizontal pipe segment in row $i+1$ from the midpoint of column $a_h$ to the midpoint of column $b_h+1$,
    \item delete the vertical pipe segment in column $a_h$ from the midpoint of row $i$ to the midpoint of row $i+1$, and
    \item draw a vertical pipe segment in column $b_h+1$ from the midpoint of row $i$ to the midpoint of row $i+1$.
\end{enumerate}
After doing these replacements, we obtain a new diagram $\mathcal B'$ which we will show is a valid BPD. Call this map $\phi_i$. Because the $[a_h,b_h+1]$'s are disjoint, we may make all of these replacements simultaneously.

\begin{proposition}
\label{prop:droppablebpd}
    The map $\phi_i$ is well-defined. Fix $A \in \asm(n)$ and $\mathcal{B}=\asmtobpd(A)$. Then $\phi_i(\mathcal B)$ is the BPD associated to the  maximum element of $\pi_i^{-1}(\pi_i(A))$ in strong order.
\end{proposition}

We will prove \cref{prop:droppablebpd} later in this section.  First we start with a lemma.

\begin{lemma}
\label{lemma:bpdreplacements_droppable}
Let $\mathcal B$ be a BPD. Suppose $[a,b]$ is a droppable $i$-interval in $\mathcal B$.

\begin{enumerate}

\item The replacement $\mathcal B\mapsto \phi_i(\mathcal B)$ restricted to $[i,i+1]\times[a,a]$ must be one of the following:
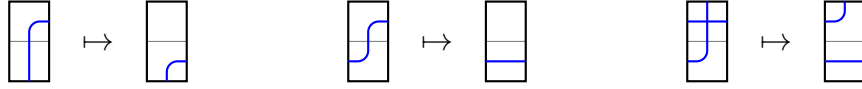
\[
\begin{tikzpicture}[x=1.25em,y=1.25em]
	\draw[step=1,gray, thin] (0,0) grid (1,2);
	\draw[thick] (0,0)--(1,0)--(1,2)--(0,2)--cycle;
	\draw[thick,rounded corners,color=blue] (.5,0)--(.5,1.5)--(1,1.5);
\end{tikzpicture}
\quad
\raisebox{1em}{$\mapsto$}
\quad
\begin{tikzpicture}[x=1.25em,y=1.25em]
	\draw[step=1,gray, thin] (0,0) grid (1,2);
	\draw[thick] (0,0)--(1,0)--(1,2)--(0,2)--cycle;
	\draw[thick,rounded corners,color=blue] (.5,0)--(.5,.5)--(1,.5);
\end{tikzpicture}
\hspace{5em}
\begin{tikzpicture}[x=1.25em,y=1.25em]
	\draw[step=1,gray, thin] (0,0) grid (1,2);
	\draw[thick] (0,0)--(1,0)--(1,2)--(0,2)--cycle;
	\draw[thick,rounded corners,color=blue] (0,.5)--(.5,.5)--(.5,1.5)--(1,1.5);
\end{tikzpicture}
\quad
\raisebox{1em}{$\mapsto$}
\quad
\begin{tikzpicture}[x=1.25em,y=1.25em]
	\draw[step=1,gray, thin] (0,0) grid (1,2);
	\draw[thick] (0,0)--(1,0)--(1,2)--(0,2)--cycle;
	\draw[thick,rounded corners,color=blue] (0,.5)--(1,.5);
\end{tikzpicture}
\hspace{5em}
\begin{tikzpicture}[x=1.25em,y=1.25em]
	\draw[step=1,gray, thin] (0,0) grid (1,2);
	\draw[thick] (0,0)--(1,0)--(1,2)--(0,2)--cycle;
	\draw[thick,rounded corners,color=blue] (0,1.5)--(1,1.5);
	\draw[thick,rounded corners,color=blue] (0,.5)--(.5,.5)--(.5,2);
\end{tikzpicture}
\quad
\raisebox{1em}{$\mapsto$}
\quad
\begin{tikzpicture}[x=1.25em,y=1.25em]
	\draw[step=1,gray, thin] (0,0) grid (1,2);
	\draw[thick] (0,0)--(1,0)--(1,2)--(0,2)--cycle;
	\draw[thick,rounded corners,color=blue] (0,1.5)--(.5,1.5)--(.5,2);
	\draw[thick,rounded corners,color=blue] (0,.5)--(1,.5);
\end{tikzpicture}
\]

 \item If $a<j<b+1$, then the replacement $\mathcal B\mapsto \phi_i(\mathcal B)$ restricted to $[i,i+1]\times[j,j]$ must be one of the following configurations:
 
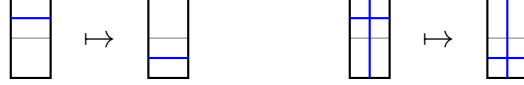
\[
 \begin{tikzpicture}[x=1.25em,y=1.25em]
	\draw[step=1,gray, thin] (0,0) grid (1,2);
	\draw[color=black, thick](0,0)rectangle(1,2);
	\draw[thick,rounded corners,color=blue](0,1.5)--(1,1.5);
	\end{tikzpicture}
\quad
 \raisebox{1em}{$\mapsto$}
 \quad
 \begin{tikzpicture}[x=1.25em,y=1.25em]
	\draw[step=1,gray, thin] (0,0) grid (1,2);
	\draw[color=black, thick](0,0)rectangle(1,2);
	\draw[thick,rounded corners,color=blue](0,.5)--(1,.5);
	\end{tikzpicture}
 \hspace{5em}
 \begin{tikzpicture}[x=1.25em,y=1.25em]
	\draw[step=1,gray, thin] (0,0) grid (1,2);
	\draw[color=black, thick](0,0)rectangle(1,2);
	\draw[thick,rounded corners,color=blue](0,1.5)--(1,1.5);
\draw[thick,rounded corners,color=blue](.5,0)--(.5,2);
	\end{tikzpicture}
\quad
 \raisebox{1em}{$\mapsto$}
 \quad
\begin{tikzpicture}[x=1.25em,y=1.25em]
	\draw[step=1,gray, thin] (0,0) grid (1,2);
	\draw[color=black, thick](0,0)rectangle(1,2);
	\draw[thick,rounded corners,color=blue](0,.5)--(1,.5);
 \draw[thick,rounded corners,color=blue](.5,0)--(.5,2);
	\end{tikzpicture}
\]

\item The replacement $\mathcal B\mapsto\phi_i(\mathcal B)$ restricted to $[i,i+1]\times [b+1,b+1]$ must be one of the following configurations:
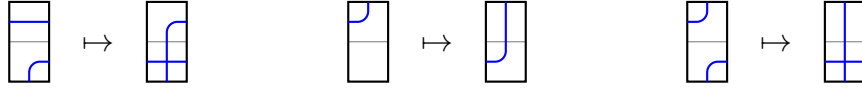
\[
\begin{tikzpicture}[x=1.25em,y=1.25em]
	\draw[step=1,gray, thin] (0,0) grid (1,2);
	\draw[thick] (0,0)--(1,0)--(1,2)--(0,2)--cycle;
	\draw[thick,rounded corners,color=blue] (0,1.5)--(1,1.5);
	\draw[thick,rounded corners,color=blue] (.5,0)--(.5,.5)--(1,.5);
\end{tikzpicture}
\quad
\raisebox{1em}{$\mapsto$}
\quad
\begin{tikzpicture}[x=1.25em,y=1.25em]
	\draw[step=1,gray, thin] (0,0) grid (1,2);
	\draw[thick] (0,0)--(1,0)--(1,2)--(0,2)--cycle;
	\draw[thick,rounded corners,color=blue] (.5,0)--(.5,1.5)--(1,1.5);
	\draw[thick,rounded corners,color=blue] (0,.5)--(1,.5);
\end{tikzpicture}
\hspace{5em}
\begin{tikzpicture}[x=1.25em,y=1.25em]
	\draw[step=1,gray, thin] (0,0) grid (1,2);
	\draw[thick] (0,0)--(1,0)--(1,2)--(0,2)--cycle;
	\draw[thick,rounded corners,color=blue] (0,1.5)--(.5,1.5)--(.5,2);
\end{tikzpicture}
\quad
\raisebox{1em}{$\mapsto$}
\quad
\begin{tikzpicture}[x=1.25em,y=1.25em]
	\draw[step=1,gray, thin] (0,0) grid (1,2);
	\draw[thick] (0,0)--(1,0)--(1,2)--(0,2)--cycle;
	\draw[thick,rounded corners,color=blue] (0,.5)--(.5,.5)--(.5,2);
\end{tikzpicture}
\hspace{5em}
\begin{tikzpicture}[x=1.25em,y=1.25em]
	\draw[step=1,gray, thin] (0,0) grid (1,2);
	\draw[thick] (0,0)--(1,0)--(1,2)--(0,2)--cycle;
	\draw[thick,rounded corners,color=blue] (0,1.5)--(.5,1.5)--(.5,2);
	\draw[thick,rounded corners,color=blue] (.5,0)--(.5,.5)--(1,.5);
\end{tikzpicture}
\quad
\raisebox{1em}{$\mapsto$}
\quad
\begin{tikzpicture}[x=1.25em,y=1.25em]
	\draw[step=1,gray, thin] (0,0) grid (1,2);
	\draw[thick] (0,0)--(1,0)--(1,2)--(0,2)--cycle;
	\draw[thick,rounded corners,color=blue] (0,.5)--(1,.5);
	\draw[thick,rounded corners,color=blue] (.5,0)--(.5,2);
\end{tikzpicture}
\]

\end{enumerate}
\end{lemma}

\begin{proof}
    \noindent (1) Because column $a$ is the start of a droppable $i$-interval, $\mathcal{B}$ must have a pipe exiting column $a$ in row $i$ and no pipe exiting column $a$ in row $i+1$. Furthermore, it must not be the case that there is a pipe entering column $a$ in row $i$ while there is no pipe entering column $a$ in row $i+1$. Otherwise, column $a-1$ would already be a droppable column crossing for $i$ in $\mathcal B$.
    
    The three pictured configurations are the only ones satisfying these requirements. The local replacement within $[i,i+1]\times [a,a]$ in all cases is as described in the definition of $\phi_i$.

        \noindent (2) In this case, for $a < j < b+1$, row $i$ of $\mathcal B$ must have a pipe that both enters and exits column $j$, while row $i+1$ must have no pipe entering or exiting column $j$. The two pictured cases  are the only configurations satisfying these requirements. The local replacement within $[i,i+1]\times [j,j]$ in both cases is as described in the definition of $\phi_i$.

    \noindent (3) Because column $b$ is the end of a droppable $i$-interval, $\mathcal{B}$ must have a pipe entering column $b+1$ in row $i$, but must not have a pipe entering column $b+1$ in row $i+1$. Furthermore, it must not be the case that there is a pipe leaving column $b+1$ in row $i$ yet no pipe leaving column $b+1$ in row $i+1$. Otherwise, column $b+1$ would also be a droppable column crossing for $i$, contradicting the maximality of the interval $[a,b]$.

    The three pictured configurations are the only ones satisfying these requirements. The local replacement within $[i,i+1]\times [b+1,b+1]$ in all cases is as described in the definition of $\phi_i$.
\end{proof}

\begin{example}
\label{example:droop}
We may visualize the replacements in the map $\phi_i$ as a droop move, as defined in \cite{LLS21}. Suppose we have a droppable $i$-interval, pictured below on the left.  We may change the cross in the top left corner into a bump, then droop the bump downwards so that it passes through the bottom right corner.  
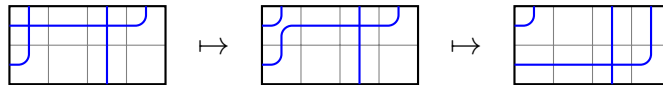
\[
\begin{tikzpicture}[x=1.25em,y=1.25em]
    \draw[step=1,gray, thin] (0,0) grid (4,2);
    \draw[thick] (0,0) rectangle (4,2);
    \draw[thick,rounded corners,color=blue] (0,1.5) -- (3.5,1.5)--(3.5,2); 
    \draw[thick,rounded corners,color=blue] (0.5,2) -- (0.5,.5)--(0,.5);
    \draw[thick,rounded corners,color=blue] (2.5,2) -- (2.5,0);
\end{tikzpicture}
\quad \raisebox{1em}{$\mapsto$} \quad
\begin{tikzpicture}[x=1.25em,y=1.25em]
    \draw[step=1,gray, thin] (0,0) grid (4,2);
    \draw[thick] (0,0) rectangle (4,2);
    \draw[thick,rounded corners,color=blue] (0,1.5) -- (0.5,1.5) -- (0.5,2); 
    \draw[thick,rounded corners,color=blue] (0,.5)--(0.5,.5) -- (0.5,1.5) -- (3.5,1.5)--(3.5,2);
    \draw[thick,rounded corners,color=blue] (2.5,2) -- (2.5,0);
\end{tikzpicture}
\quad \raisebox{1em}{$\mapsto$} \quad
\begin{tikzpicture}[x=1.25em,y=1.25em]
    \draw[step=1,gray, thin] (0,0) grid (4,2);
    \draw[thick] (0,0) rectangle (4,2);
    \draw[thick,rounded corners,color=blue] (0,1.5) -- (0.5,1.5) -- (0.5,2); 
    \draw[thick,rounded corners,color=blue] (0,0.5) -- (3.5,0.5) -- (3.5,2); 
    \draw[thick,rounded corners,color=blue] (2.5,2) -- (2.5,0);
\end{tikzpicture}
\]
This produces a valid BPD.
\end{example}

We now prove the analogous result to \cref{prop:weakorderbpd}.

\begin{proof}[{Proof of \cref{prop:droppablebpd}}]

Well-definedness follows from \cref{lemma:bpdreplacements_droppable}.
Additionally, we have that $\dc_i(\phi_i(\mathcal B))=\emptyset$ by \cref{lemma:bpdreplacements_droppable}.  In \cref{prop:moves-to-get-preimage}, each of the given moves involves a droppable column crossing for $i$.  Because $\dc_i(\phi_i(\mathcal B))=\emptyset$, there are no moves that one can apply in row $i$ of $\phi_i(\mathcal B)$ that would produce a new BPD $\mathcal B'$ with $\asmtobpd^{-1}(\phi_i(\mathcal B))<\asmtobpd^{-1}(\mathcal B')$.  So $\asmtobpd^{-1}(\phi_i(\mathcal B))$ must be the maximum element of $\pi_i^{-1}(\pi_i(A))$.
\end{proof}

\begin{example}
We continue with the interval from \cref{example:droop}.  Pictured below is the poset induced on the intermediate diagrams by applying the moves in \cref{prop:moves-to-get-preimage}.

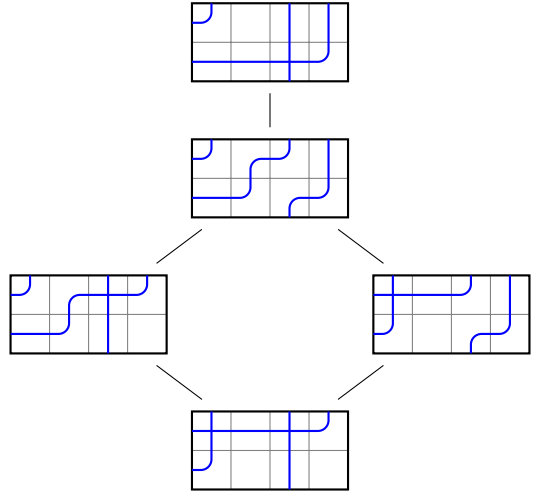
\begin{center}
\begin{tikzpicture}[scale=1.2]
    \node (A) at (0,4.5) {
\begin{tikzpicture}[x=1.25em,y=1.25em]
    \draw[step=1,gray, thin] (0,0) grid (4,2);
    \draw[thick] (0,0) rectangle (4,2);
    \draw[thick,rounded corners,color=blue] (0,.5)--(3.5,.5)--(3.5,2); 
    \draw[thick,rounded corners,color=blue] (0.5,2) -- (0.5,1.5)--(0,1.5);
    \draw[thick,rounded corners,color=blue] (2.5,2) -- (2.5,0);
\end{tikzpicture}};
    
    \node (B) at (0,3) {\begin{tikzpicture}[x=1.25em,y=1.25em]
    \draw[step=1,gray, thin] (0,0) grid (4,2);
    \draw[thick] (0,0) rectangle (4,2);
     \draw[thick,rounded corners,color=blue] (0,.5)--(1.5,.5)--(1.5,1.5) -- (2.5,1.5)--(2.5,2); 
    \draw[thick,rounded corners,color=blue] (0.5,2) -- (0.5,1.5)--(0,1.5);
    \draw[thick,rounded corners,color=blue] (3.5,2)--(3.5,.5)--(2.5,.5) -- (2.5,0);
\end{tikzpicture}};
    
    \node (C) at (2,1.5) {\begin{tikzpicture}[x=1.25em,y=1.25em]
    \draw[step=1,gray, thin] (0,0) grid (4,2);
    \draw[thick] (0,0) rectangle (4,2);
    \draw[thick,rounded corners,color=blue] (0,1.5) -- (2.5,1.5)--(2.5,2); 
    \draw[thick,rounded corners,color=blue] (0.5,2) -- (0.5,.5)--(0,.5);
    \draw[thick,rounded corners,color=blue] (3.5,2)--(3.5,.5)--(2.5,.5) -- (2.5,0);
\end{tikzpicture}
};
    \node (D) at (-2,1.5) {\begin{tikzpicture}[x=1.25em,y=1.25em]
    \draw[step=1,gray, thin] (0,0) grid (4,2);
    \draw[thick] (0,0) rectangle (4,2);
    \draw[thick,rounded corners,color=blue] (0,.5)--(1.5,.5)--(1.5,1.5) -- (3.5,1.5)--(3.5,2); 
    \draw[thick,rounded corners,color=blue] (0.5,2) -- (0.5,1.5)--(0,1.5);
    \draw[thick,rounded corners,color=blue] (2.5,2) -- (2.5,0);
\end{tikzpicture}
};
    
    \node (E) at (0,0) {\begin{tikzpicture}[x=1.25em,y=1.25em]
    \draw[step=1,gray, thin] (0,0) grid (4,2);
    \draw[thick] (0,0) rectangle (4,2);
    \draw[thick,rounded corners,color=blue] (0,1.5) -- (3.5,1.5)--(3.5,2); 
    \draw[thick,rounded corners,color=blue] (0.5,2) -- (0.5,.5)--(0,.5);
    \draw[thick,rounded corners,color=blue] (2.5,2) -- (2.5,0);
\end{tikzpicture}};
    \draw (A) -- (B);    
    \draw (B) -- (C);
    \draw (B) -- (D);
    \draw (C) -- (E);
    \draw (D) -- (E);
\end{tikzpicture}
\end{center}
\end{example}

\section*{Acknowledgements}

This project started as part of the Virtual Workshop for Women in Commutative Algebra and Algebraic Geometry, sponsored by the Fields Institute and organized by Megumi Harada and Claudia Miller. The authors would like to thank Zach Hamaker and Vic Reiner for helpful discussions on weak order on ASMs. They thank Daoji Huang for the question giving rise to \cref{rmk:Daoji's Question}.

\bibliographystyle{amsalpha} 
\bibliography{asm.bib}

\providecommand{\bysame}{\leavevmode\hbox to3em{\hrulefill}\thinspace}
\providecommand{\MR}{\relax\ifhmode\unskip\space\fi MR }
\providecommand{\MRhref}[2]{%
  \href{http://www.ams.org/mathscinet-getitem?mr=#1}{#2}
}
\providecommand{\href}[2]{#2}
\begin{thebibliography}{EKW25}

\bibitem[APP26]{APP}
David Anderson, Greta Panova, and Leonid Petrov, \emph{Computation and sampling
  for {S}chubert specializations}, Preprint (2026), 33 pages, {\sf
  arXiv:2603.20104}.

\bibitem[BB05]{BB05}
Anders Bj\"{o}rner and Francesco Brenti, \emph{Combinatorics of {C}oxeter
  groups}, Graduate Texts in Mathematics, vol. 231, Springer, New York, 2005.

\bibitem[Beh08]{Beh08}
Roger~E. Behrend, \emph{Osculating paths and oscillating tableaux}, Electron.
  J. Combin. \textbf{15} (2008), no.~1, Research Paper 7, 60.

\bibitem[BSS26]{BSS}
Mathilde Bouvel, Rebecca Smith, and Jessica Striker, \emph{Key-avoidance for
  alternating sign matrices}, Discrete Math. Theor. Comput. Sci. \textbf{27}
  ([2025--2026]), no.~1, Paper No. 3, 28.

\bibitem[EKW25]{EKW-main}
Laura Escobar, Patricia Klein, and Anna Weigandt, \emph{Algebra and geometry of
  {ASM} weak order}, Preprint (2025), 33 pages, {\sf arXiv:2502.19266}.

\bibitem[For08]{For08}
Marc Fortin, \emph{The {M}ac{N}eille completion of the poset of partial
  injective functions}, Electron. J. Combin. \textbf{15} (2008), no.~1,
  Research paper 62, 30.

\bibitem[Ful92]{Ful92}
William Fulton, \emph{Flags, {S}chubert polynomials, degeneracy loci, and
  determinantal formulas}, Duke Math. J. \textbf{65} (1992), no.~3, 381--420.

\bibitem[HR20]{HR20}
Zachary Hamaker and Victor Reiner, \emph{Weak order and descents for monotone
  triangles}, European J. Combin. \textbf{86} (2020), 103083, 22.

\bibitem[KM05]{KM05}
Allen Knutson and Ezra Miller, \emph{Gr\"{o}bner geometry of {S}chubert
  polynomials}, Ann. of Math. (2) \textbf{161} (2005), no.~3, 1245--1318.

\bibitem[Knu09]{Knu09}
Allen Knutson, \emph{Frobenius splitting, point-counting, and degeneration},
  Preprint (2009), 28 pages, {\sf arXiv:0911.4941}.

\bibitem[KW23]{KW23}
Patricia Klein and Anna Weigandt, \emph{Bumpless pipe dreams encode {G}r\"obner
  geometry of {S}chubert polynomials}, Preprint (2023), 42 pages, {\sf
  arXiv:2108.08370}.

\bibitem[Las02]{Las-Ice}
Alain Lascoux, \emph{Chern and {Y}ang through ice}, Preprint (2002), 17 pages,
  tinyurl.com/y64bnpro.

\bibitem[LLS21]{LLS21}
Thomas Lam, Seung~Jin Lee, and Mark Shimozono, \emph{Back stable {S}chubert
  calculus}, Compos. Math. \textbf{157} (2021), no.~5, 883--962.

\bibitem[LS96]{LS96}
Alain Lascoux and Marcel-Paul Sch\"{u}tzenberger, \emph{Treillis et bases des
  groupes de {C}oxeter}, Electron. J. Combin. \textbf{3} (1996), no.~2,
  Research paper 27, approx. 35.

\bibitem[MRR83]{MRR83}
W.~H. Mills, David~P. Robbins, and Howard Rumsey, Jr., \emph{Alternating sign
  matrices and descending plane partitions}, J. Combin. Theory Ser. A
  \textbf{34} (1983), no.~3, 340--359.

\bibitem[Wei17]{Wei17}
Anna Weigandt, \emph{Prism tableaux for alternating sign matrix varieties},
  Preprint (2017), 33 pages, {\sf arXiv:1708.07236}.

\bibitem[Wei21]{Wei21}
\bysame, \emph{Bumpless pipe dreams and alternating sign matrices}, J. Combin.
  Theory Ser. A \textbf{182} (2021), Paper No. 105470, 52.

\end{thebibliography}
\end{document}